\documentclass[a4paper,oneside]{amsart}
\usepackage[utf8]{inputenc}
\usepackage[a4paper,margin=3cm]{geometry} %showframe=true
\usepackage{bm}
\usepackage{amsmath}
\usepackage{amsthm}
\usepackage{amscd}
\usepackage{amssymb}
\usepackage{MnSymbol}
\usepackage{latexsym}
\usepackage{eucal}
\usepackage{dsfont}
\usepackage{mathtools}
\usepackage{enumitem}
\usepackage{todonotes}

\usepackage{verbatim}
\usepackage{graphicx}
\usepackage{float}
\usepackage{array,booktabs}

\usepackage{pdflscape}

\usepackage[colorlinks,pdftex]{hyperref}
\hypersetup{
linkcolor=black,
citecolor=black,
pdftitle={},
pdfauthor={Franziska K\"uhn},
pdfkeywords={},
}

\makeatletter
\renewcommand\section{\@startsection{section}{1}{0mm}{-1.5\baselineskip}{\baselineskip}{\normalsize\bfseries\sffamily}}
\renewcommand\subsection{\@startsection{subsection}{1}{0mm}{-\baselineskip}{\baselineskip}{\normalsize\bfseries\sffamily}}
\makeatother

\makeatletter
\def\@fnsymbol#1{\ensuremath{\ifcase#1\or *\or **\or \dagger\or \ddagger\or
   \mathsection\or \mathparagraph\or \|\or \dagger\dagger
   \or \ddagger\ddagger \else\@ctrerr\fi}}

\newlength{\preskip}
\newlength{\postskip}
\makeatother

\newtheoremstyle{theorem}{\preskip}{\postskip}{\itshape}{}{\bfseries}{}
{.5em}{\textbf{\thmname{#1}\thmnumber{ #2} (\thmnote{ #3})}}
\newtheoremstyle{definition}{\preskip}{\postskip}{\normalfont}{0pt}{\bfseries}{}{.5em}{}
\newtheoremstyle{remark}{\preskip}{\postskip}{\normalfont}{0pt}{\bfseries}{}{.5em}{}

\swapnumbers
\theoremstyle{theorem} \newtheorem{thm}{Theorem}[section]
\theoremstyle{theorem} \newtheorem{lem}[thm]{Lemma}
\theoremstyle{theorem} \newtheorem{prop}[thm]{Proposition}
\theoremstyle{theorem} \newtheorem{kor}[thm]{Corollary}
\theoremstyle{definition} \newtheorem{defn}[thm]{Definition}
\theoremstyle{remark} \newtheorem{bem_thm}[thm]{Remark}
\theoremstyle{remark} 
\theoremstyle{definition} 
\theoremstyle{definition} \newtheorem*{ack}{Acknowledgements}
\theoremstyle{remark} \newtheorem{bem}[thm]{Remark}
\theoremstyle{remark} 
\theoremstyle{definition}  \newtheorem{bsp}[thm]{Example}
\theoremstyle{definition}  
\theoremstyle{definition} \newtheorem*{thmo}{Theorem}

\DeclareMathOperator \re {Re}

\DeclareMathOperator \sgn {sgn}

\DeclareMathOperator \spt {supp}

\DeclareMathOperator \tr {tr}

\newcommand{\I}{\mathds{1}}

\newcommand\fa{\qquad \text{for all \ }}

\newcommand\mc[1] {\mathcal{#1}}
\newcommand\mbb[1] {\mathds{#1}}

\newcommand{\eps}{\varepsilon}

\begin{document}

\title[Existence of (Markovian) solutions to martingale problems]{Existence of (Markovian) solutions to martingale problems associated with L\'evy-type operators}
\author[F.~K\"{u}hn]{Franziska K\"{u}hn} 
\address[F.~K\"{u}hn]{Institut de Math\'ematiques de Toulouse, Universit\'e Paul Sabatier III Toulouse, 118 Route de Narbonne, 31062 Toulouse, France}
\email{franziska.kuhn@math.univ-toulouse.fr}

\begin{abstract}
Let $A$ be a pseudo-differential operator with symbol $q(x,\xi)$. In this paper we derive sufficient conditions which ensure the existence of a solution to the $(A,C_c^{\infty}(\mathbb{R}^d))$-martingale problem. If the symbol $q$ depends continuously on the space variable $x$, then the existence of solutions is well understood, and therefore the focus lies on martingale problems for pseudo-differential operators with discontinuous coefficients. We prove an existence result which allows us, in particular, to obtain new insights on the existence of weak solutions to a class of L\'evy-driven SDEs with Borel measurable coefficients and on the the existence of stable-like processes with discontinuous coefficients. Moreover, we establish a Markovian selection theorem which shows that -- under mild assumptions -- the $(A,C_c^{\infty}(\mathbb{R}^d))$-martingale problem gives rise to a strong Markov process. The result applies, in particular, to L\'evy-driven SDEs. We illustrate the Markovian selection theorem with applications in the theory of non-local operators and equations; in particular, we establish under weak regularity assumptions a Harnack inequality for non-local operators of variable order.
\end{abstract}
\maketitle

L\'evy-type operators appear naturally in the theory of stochastic processes, for instance as infinitesimal generators of L\'evy(-type) processes \cite{ltp,jac2} and in the context of stochastic differential equations \cite{sde,kurtz}. A L\'evy-type operator is defined on the smooth functions with compact support $C_c^{\infty}(\mbb{R}^d)$ and has a representation of the form
\begin{align*}
	Af(x) &= b(x) \cdot \nabla f(x) + \frac{1}{2} \tr(Q(x) \cdot \nabla^2 f(x)) \\
	&\quad + \int_{\mbb{R}^d \backslash \{0\}} \left( f(x+y)-f(x)-\nabla f(x) \cdot y \I_{(0,1)}(|y|) \right) \, \nu(x,dy)
\end{align*}
where $(b(x),Q(x),\nu(x,dy))$ is for each fixed $x \in \mbb{R}^d$ a L\'evy triplet. Equivalently, $A$ can be written as a pseudo-differential operator \begin{equation*}
	Af(x) =- \int_{\mbb{R}^d} e^{ix \cdot \xi} q(x,\xi) \hat{f}(\xi) \, d\xi 
\end{equation*}
with symbol $q$, \begin{equation*}
	q(x,\xi) := - ib(x) \cdot \xi + \frac{1}{2} \xi \cdot Q(x) \xi + \int_{\mbb{R}^d \backslash \{0\}} \left( 1-e^{iy \cdot \xi} + iy \cdot \xi \I_{(0,1)}(|y|) \right) \, \nu(x,dy).
\end{equation*}
In this paper, we are interested in the martingale problem associated with the L\'evy-type operator, i.\,e.\ for a given initial distribution $\mu$ we study probability measures $\mbb{P}^{\mu}$ on the Skorohod space $D[0,\infty)$ such that the canonical process $(X_t)_{t \geq 0}$ satisfies $\mbb{P}^{\mu}(X_0 \in \cdot) = \mu$ and \begin{equation*}
	M_t := f(X_t)-f(X_0)- \int_0^t Af(X_s) \, ds, \qquad t \geq 0,
\end{equation*}
is a $\mbb{P}^{\mu}$-martingale for any $f \in C_c^{\infty}(\mbb{R}^d)$; as usual we set $\mbb{P}^x := \mbb{P}^{\delta_x}$. If the martingale problem is well-posed (i.\,e.\ $\mbb{P}^{\mu}$ is unique for any initial distribution $\mu$), then this gives a lot of additional information on the stochastic process; for instance, well-posedness of the martingale problem implies the Markov property of $(X_t)_{t \geq 0}$, see e.\,g.\  \cite[Theorem 4.4.2]{ethier}, and under some weak additional assumptions $(X_t)_{t \geq 0}$ is a Feller process, cf.\ \cite{mp}. It is, however, in general difficult to prove the well-posedness of the martingale problem (see e.\,g.\ \cite{ltp,matters} for a survey on known results), and for many interesting examples it is known that well-posedness does not hold. It is therefore of great interest to study properties of solutions to martingale problems which are not necessarily well-posed. This paper has three parts. \par
Firstly, we are interested in finding sufficient conditions on the operator $A$ (or its symbol $q$) which ensure the existence of a solution to the $(A,C_c^{\infty}(\mbb{R}^d))$-martingale problem. If $q$ has continuous coefficients, i.\,e.\ $x \mapsto q(x,\xi)$ is continuous for all $\xi$, then the existence of solutions is well understood, cf.\ \cite[Theorem 3.2]{hoh95} and \cite[Corollary 3.2]{mp}. The situation is more delicate if $q$ is discontinuous, and we are not aware of a general existence result in the discontinuous setting (see Section~\ref{exi} for a detailed discussion of known results). In this paper, we will show that a solution to the martingale problem for a L\'evy-type operator $A$ (with possibly discontinuous coefficients) exists if $A$ can be approximated by a sequence of L\'evy-type operators $A_n$, $n \geq 1$, satisfying a Krylov estimate, cf.\ Theorem~\ref{exi-9} for the precise statement. Combining the  result with heat kernel estimates obtained in \cite{matters}, we obtain a new existence result for weak solutions to L\'evy-driven SDEs with Borel measurable coefficients, cf.\ Corollary~\ref{exi-13}. Moreover, Theorem~\ref{exi-9} allows us to prove the existence of stable-like processes with discontinuous coefficients, cf.\ Example~\ref{exi-17} and Example~\ref{exi-19}. \par
Secondly, we will study under which assumptions a solution to the martingale problem gives rise to a (strong) Markov process. More precisely, we will investigate the following question: Assuming that for each initial distribution $\mu$ there exists a solution to the $(A,C_c^{\infty}(\mbb{R}^d))$-martingale problem with initial distribution $\mu$, i.\,e.\ \begin{equation*}
	\Pi_{\mu} := \{\mbb{P}; \text{$\mbb{P}$ is a solution to the martingale problem with initial distribution $\mu$}\} \neq \emptyset,
\end{equation*}
then under which assumptions can we choose $\mbb{P}^x \in \Pi_{\delta_x}$ such that $(X_t,\mbb{P}^x; x \in \mbb{R}^d, t \geq 0)$ is a strong Markov process? Krylov \cite{krylov} proved an abstract criterion for the existence of a Markovian selection for a large class of operators $A$ (which need not be L\'evy-type operators) and applied it to establish a Markovian selection theorem for diffusions (i.\,e.\ $A$ is a local L\'evy-type operator, $\nu=0$). Krylov's criterion has been refined by Ethier \& Kurtz, cf.\ \cite[Section 4.5]{ethier}; roughly speaking, they show that, for ``nice'' operators $A$, a certain compact containment condition implies the existence of a Markovian selection. The result is the key tool to prove a Markovian selection theorem for L\'evy-type operators; in particular, we obtain the following statement, which seems to be new.

\begin{thmo}
	Let $A$ be a L\'evy-type operator with symbol $q$. If $x \mapsto q(x,\xi)$ is continuous for all $\xi \in \mbb{R}^d$, $q$ is locally bounded, i.\,e.\ \begin{equation*}
		\forall R>0: \quad \sup_{|x| \leq R} \sup_{|\xi| \leq 1} |q(x,\xi)|< \infty,
	\end{equation*}
	and $q$ is locally uniformly continuous at $\xi=0$
	\begin{equation*}
		\lim_{R \to \infty} \sup_{|y| \leq R} \sup_{|\xi| \leq R^{-1}} |q(y,\xi)| = 0, 
	\end{equation*}
	then there exists a conservative strong Markov process $(X_t,\mc{F}_t,\mbb{P}^x; x \in \mbb{R}^d, t \geq 0)$ such that $\mbb{P}^x$ is, for each $x \in \mbb{R}^d$, a solution to the $(A,C_c^{\infty}(\mbb{R}^d))$-martingale problem with initial distribution $\mu = \delta_x$.
\end{thmo}

If the symbol $q$ does not have continuous coefficients, we have to assume additionally the existence of a solution to the martingale problem for any initial distribution $\mu$, cf.\ Theorem~\ref{cont-1} for details. As a by-product, we obtain a sufficient condition for the existence of a Markovian (weak) solution to L\'evy-driven SDEs, cf.\ Corollary~\ref{cont-3}. \par
In the third, and final, part of the paper, we will illustrate the well-established fact that there is a strong connection between probability theory and the analysis of PDEs and pseudo-differential operators. We will present two applications of Markovian selection theorems in the theory of non-local operators and equations. The first one is a Harnack inequality for a class of pseudo-differential operators, cf.\ Section~\ref{har}, and the second one concerns viscosity solutions to a certain integro-differential equation, cf.\ Section~\ref{vis}.  \par \medskip

%The paper is organized as follows. In Section~\ref{def} we introduce basic definitions and notation. Section~\ref{exi} is concerned with sufficient conditions for the existence of solutions to martingale problems with discontinuous coefficients. In Section~\ref{cont} we establish the Markovian selection theorem for martingale problems associated with L\'evy-type operators. Applications of the Markovian selection theorems are presented in Section~\ref{ana}.

\section{Preliminaries} \label{def}

We consider $\mbb{R}^d$ endowed with the Borel $\sigma$-algebra $\mc{B}(\mbb{R}^d)$ and write $B(x,r)$ for the open ball centered at $x \in \mbb{R}^d$ with radius $r>0$; $\mbb{R}^d_{\partial}$ is the one-point compactification of $\mbb{R}^d$. The transpose of a matrix $A \in \mbb{R}^{d \times d}$ is denoted by $A^T$. If a certain statement holds for $x \in \mbb{R}^d$ with $|x|$ sufficiently large, we write ``for $|x| \gg 1$''. We denote by $C(\mbb{R}^d)$ the space of continuous functions $f: \mbb{R}^d \to \mbb{R}$; $C_{\infty}(\mbb{R}^d)$ (resp.\ $C_b(\mbb{R}^d)$) is the space of continuous functions which vanish at infinity (resp.\ are bounded). A function $f: [0,\infty) \to \mbb{R}^d$ is in the Skorohod space $D[0,\infty)$ if $f$ is right-continuous and has finite left-hand limits in $\mbb{R}^d$. On  $C_b^2(\mbb{R}^d)$, the space of two times continuously differentiable functions which are bounded (with its derivatives), we define a norm by \begin{equation*}
	\|f\|_{(2)} := \|f\|_{\infty} + \|\nabla f \|_{\infty} + \|\nabla^2 f\|_{\infty}, \qquad f \in C_b^2(\mbb{R}^d),
\end{equation*}
here $\nabla f$ and $\nabla^2 f$ are the gradient and Hessian of $f$, respectively. We write \begin{equation*}
	\|f\|_{\varrho} := \|f\|_{\infty}+ \sup_{\substack{x,y \in \mbb{R}^d \\ x \neq y}} \frac{|f(x)-f(y)|}{|x-y|^{\varrho}}, \qquad \varrho \in (0,1]
\end{equation*}
for the H\"older norm of a function $f$. The space of bounded Borel measurable functions $f: \mbb{R}^d \to \mbb{R}$ is denoted by $\mc{B}_b(\mbb{R}^d)$, and $\mc{P}(\mbb{R}^d)$ is the family of probability measures on $(\mbb{R}^d,\mc{B}(\mbb{R}^d))$. \par
For a filtration $(\mc{F}_t)_{t \geq 0}$ on a measurable space $(\Omega,\mc{A})$ we set $\mc{F}_{\infty} := \sigma(\mc{F}_t; t \geq 0)$. If $\tau: \Omega \to [0,\infty]$ is an $\mc{F}_t$-stopping time, i.\,e.\ $\{\tau \leq t\} \in \mc{F}_t$ for all $t \geq 0$, then \begin{equation*}
	\mc{F}_{\tau} := \{A \in \mc{F}_{\infty}: \forall t \geq 0: \, \, A \cap \{\tau \leq t\} \in \mc{F}_t\}
\end{equation*}
is the $\sigma$-algebra associated with $\tau$. For a probability measure $\mbb{P}$ on $(\Omega,\mc{A})$ and a bounded $\mc{A}$-measurable random variable $Y$ we denote by \begin{equation*}
	\mbb{E} Y := \int_{\Omega} Y(\omega) \, \mbb{P}(d\omega)
\end{equation*}
the expectation with respect to $\mbb{P}$; we write $\mbb{E}_{\mbb{P}}$ if we need to emphasize the underlying probability measure $\mbb{P}$.  \par
We will usually work on the Skorohod space $\Omega = D[0,\infty)$ endowed with the Borel $\sigma$-algebra induced by the Skorohod topology. We denote by $X_t(\omega) := \omega(t)$, $\omega \in D[0,\infty)$, the canonical process. Unless otherwise mentioned, we will always consider the canonical filtration $\mc{F}_t := \mc{F}_t^X := \sigma(X_s; s \leq t)$ of $(X_t)_{t \geq 0}$. \par
If $\mu \in \mc{P}(\mbb{R}^d)$ is a probability measure and $(A,\mc{D})$ a linear operator with domain $\mc{D} \subseteq \mc{B}_b(\mbb{R}^d)$, then we say that a probability measure $\mbb{P}^{\mu}$ on $\Omega=D[0,\infty)$ is a \emph{solution to the $(A,\mc{D})$-martingale problem with initial distribution $\mu$} if $\mbb{P}^{\mu}(X_0 \in \cdot) = \mu(\cdot)$ and \begin{equation*}
	M_t^u := u(X_t)-u(X_0)- \int_0^t Au(X_s) \, ds, \qquad t \geq 0,
\end{equation*}
is a $\mbb{P}^{\mu}$-martingale with respect to the canonical filtration $(\mc{F}_t)_{t \geq 0}$ for all $u \in \mc{D}$. Note that our definition entails, in particular, that $(X_t)_{t \geq 0}$ does $\mbb{P}^{\mu}$-almost surely not explode in finite time, i.\,e.\ we consider only conservative solutions to the martingale problem. We write $\Pi_{\mu}$ for the family of solutions to the $(A,\mc{D})$-martingale problem with initial distribution $\mu$. If $\mu = \delta_x$ is a Dirac distribution, then we use the shorthand $\mbb{P}^x := \mbb{P}^{\delta_x}$ and $\Pi_x := \Pi_{\delta_x}$. The $(A,\mc{D})$-martingale problem is \emph{well-posed} if for any $\mu \in \mc{P}(\mbb{R}^d)$ there exists a unique solution to the $(A,\mc{D})$-martingale problem with initial distribution $\mu$. For a comprehensive study of martingale problems see \cite[Chapter 4]{ethier}. \par
In this paper we are interested in martingale problems associated with \emph{pseudo-differential operators} (also called \emph{L\'evy-type operators}), that is, operators of the form \begin{equation}
\label{pseudo}	Af(x) := (-q(x,D)f)(x) := - \int_{\mbb{R}^d} e^{ix \cdot \xi} q(x,\xi) \hat{f}(\xi) \, d\xi, \qquad f \in \mc{D} :=C_c^{\infty}(\mbb{R}^d), \, \, x \in \mbb{R}^d,
\end{equation}
where $\hat{f}(\xi) := (2\pi)^{-d} \int_{\mbb{R}^d} e^{-ix \cdot \xi} f(x) \, dx$ denotes the Fourier transform of $f$ and \begin{equation}
	q(x,\xi) := q(x,0)- ib(x) \cdot \xi + \frac{1}{2} \xi \cdot Q(x) \xi + \int_{\mbb{R}^d \backslash \{0\}} \left( 1-e^{iy \cdot \xi} + iy \cdot \xi \I_{(0,1)}(|y|) \right) \, \nu(x,dy) \label{symbol}
\end{equation}
is the \emph{symbol} of the pseudo-differential operator $A$. For each fixed $x \in \mbb{R}^d$, $q(x,\cdot)$ is a \emph{continuous negative definite function} and $(b(x),Q(x),\nu(x,dy))$ is a \emph{L\'evy triplet}, i.\,e.\ $b(x) \in \mbb{R}^d$, $Q(x) \in \mbb{R}^{d \times d}$ is a symmetric positive semidefinite matrix and $\nu(x,dy)$ is a $\sigma$-finite measure on $(\mbb{R}^d \backslash \{0\},\mc{B}(\mbb{R}^d \backslash \{0\}))$ satisfying $\int_{y \neq 0} \min\{1,|y|^2\} \,\nu(x,dy)<\infty$. We call $(b,Q,\nu)$ the \emph{characteristics} of $q$. Using properties of the Fourier transform, it is not difficult to see that  \eqref{pseudo} is equivalent to \begin{align*}
	Af(x) &= q(x,0) f(x) + b(x) \cdot \nabla f(x) + \frac{1}{2} \tr(Q(x) \cdot \nabla^2 f(x)) \\
	&\quad + \int_{\mbb{R}^d \backslash \{0\}} \left( f(x+y)-f(x)-\nabla f(x) \cdot y \I_{(0,1)}(|y|) \right) \, \nu(x,dy); 
\end{align*}
here $\tr(Q(x) \cdot \nabla^2 f(x))$ denotes the trace of the matrix $Q(x) \cdot \nabla^2 f(x)$. Throughout this paper, we will always assume that $q(x,0)=0$ for all $x \in \mbb{R}^d$ and that $(x,\xi) \mapsto q(x,\xi)$ is Borel measurable. A symbol $q$ with characteristics $(b,Q,\nu)$ is \emph{locally bounded} if\begin{equation}
	\sup_{x \in K}  \left( |b(x)| + |Q(x)| + \int_{y \neq 0} \min\{|y|^2,1\} \, \nu(x,dy) \right) < \infty \label{bdd-coeff}
\end{equation}
for any compact set $K \subseteq \mbb{R}^d$; by \cite[Lemma 6.2]{schnurr}, $q$ is locally bounded if, and only if, for any $R>0$ there exists a finite constant $C_R>0$ such that $|q(x,\xi)| \leq c_R (1+|\xi|^2)$ for all $|x| \leq R$, $\xi \in \mbb{R}^d$. If \eqref{bdd-coeff} holds for $K=\mbb{R}^d$, then $q$ has \emph{bounded coefficients}. We say that $q$ has \emph{continuous coefficients} if $x \mapsto q(x,\xi)$ is continuous for all $\xi \in \mbb{R}^d$, see \cite[Theorem A.1]{timechange} for a characterization in terms of the characteristics $(b,Q,\nu)$. Our standard references for martingale problems associated with pseudo-differential operators is the monograph \cite{jac3}, see also \cite{hoh} and the references therein. \par
There is a close connection between Feller processes and martingale problems for pseudo-differential operators, cf.\ \cite{ltp,mp} for a detailed discussion. If the symbol $q$ is of the form \begin{equation*}
	q(x,\xi) = -ib(x) \cdot \xi + \psi(\sigma(x)^T \xi), \qquad x,\xi \in \mbb{R}^d
\end{equation*}
for the characteristic exponent $\psi$ of some L\'evy process $(L_t)_{t \geq 0}$, it is known that a solution to the $(-q(x,D),C_c^{\infty}(\mbb{R}^d))$-martingale problem gives rise to a weak solution to the L\'evy-driven SDE \begin{equation}
	dX_t = b(X_{t-}) \, dt + \sigma(X_{t-}) \, dL_t \label{sde}
\end{equation}
and vice versa, cf.\ \cite{kurtz}.

\section{Existence of solutions to martingale problems with discontinuous coefficients} \label{exi}

Let $(q(x,\cdot))_{x \in \mbb{R}^d}$ be a family of continuous negative definite functions represented by \eqref{symbol} such that $q(x,0)=0$ for all $x \in \mbb{R}^d$. If $x \mapsto q(x,\xi)$ is continuous, then there are general existence results for solutions to the $(-q(x,D),C_c^{\infty}(\mbb{R}^d))$-martingale problem. The key tool is the following statement, cf.\ \cite[Theorem 4.5.4]{ethier}.

\begin{thm} \label{exi-3}
	Let $A: \mc{D}(A) \to C_{\infty}(\mbb{R}^d)$ be a linear operator such that $\mc{D}(A) \subseteq C_{\infty}(\mbb{R}^d)$, and let $\mu \in \mc{P}(\mbb{R}^d)$ be a probability measure. If $A$ satisfies the positive maximum principle and $\mc{D}(A)$ is dense in $C_{\infty}(\mbb{R}^d)$, then there exists an $\mbb{R}^d_{\partial}$-valued solution to the $(A,\mc{D}(A))$-martingale problem with initial distribution $\mu$.
\end{thm}

If $A$ is a pseudo-differential operator with symbol $q$ and domain $\mc{D}(A) := C_c^{\infty}(\mbb{R}^d)$, then the assumption $A: C_c^{\infty}(\mbb{R}^d) \to C_{\infty}(\mbb{R}^d)$ in Theorem~\ref{exi-3} means, in particular, that $x \mapsto q(x,\xi)$ has to be continuous, i.\,e.\ Theorem~\ref{exi-3} allows us only to derive existence results for martingale problems with continuous coefficients. Hoh \cite[Theorem 3.15]{hoh} used Theorem~\ref{exi-3} to establish the existence of solutions to the $(A,C_c^{\infty}(\mbb{R}^d))$-martingale problem under the assumption that $x \mapsto q(x,\xi)$ is continuous and $q$ has bounded coefficients, i.\,e.\ $|q(x,\xi)| \leq c(1+|\xi|^2)$, $x,\xi \in \mbb{R}^d$, for some absolute constant $c>0$. The following refinement has recently been obtained in \cite[Corollary 3.2]{mp}.

\begin{thm} \label{exi-1}
	Let $A$ be a pseudo-differential operator with symbol $q$, $q(x,0)=0$. If $q$ has continuous coefficients, is locally bounded and satisfies the linear growth condition \begin{equation}
	\lim_{|x| \to \infty} \sup_{|\xi| \leq |x|^{-1}} |q(x,\xi)| < \infty \label{exi-eq2}
	\end{equation}
	then there exists for any $\mu \in \mc{P}(\mbb{R}^d)$ a (non-explosive) solution to the $(A,C_c^{\infty}(\mbb{R}^d))$-martingale problem with initial distribution $\mu$.
\end{thm}

Let us mention that the growth condition \eqref{exi-eq2} can be formulated in terms of the characteristics $(b,Q,\nu)$ of $q$, cf.\ \cite[Lemma 3.1]{mp}. \par \medskip

For martingale problems with \emph{dis}continuous coefficients we are not aware of general statements on the existence of solutions. The publication \cite{imkeller} is concerned with such an existence result but, unfortunately, there seems to be a doubt about its proof.  For the particular case that the symbol $q$ of the pseudo-differential operator $A$ is of the form \begin{equation*}
	q(x,\xi) = -ib(x) \cdot \xi + \psi(\sigma(x)^T \xi), \qquad x,\xi \in \mbb{R}^d
\end{equation*}
for the characteristic exponent $\psi$ of a L\'evy process $(L_t)_{t \geq 0}$, it is known that solving the $(A,C_c^{\infty}(\mbb{R}^d))$ is equivalent to studying weak solutions to the SDE \begin{equation}
	dX_t = b(X_{t-}) \, dt + \sigma(X_{t-}) \, dL_t. \label{exi-eq3}
\end{equation}
There are, however, only few results on the existence of weak solutions to SDEs with discontinuous coefficients $b$, $\sigma$, and they are mostly restricted to SDEs driven by isotropic $\alpha$-stable L\'evy processes. Kurenok \cite{kur12} used a timechange method to study SDEs of the form \eqref{exi-eq3} driven by a one-dimensional isotropic $\alpha$-stable L\'evy process, $\alpha \in [1,2]$, and for Borel measurable coefficients $b$, $\sigma$. For the particular case that there is no drift part (i.\,e.\ $b:=0$) and $(L_t)_{t \geq 0}$ is a one-dimensional isotropic L\'evy process, Zanzotto \cite{zan02} obtained an Engelbert-Schmidt-type result which gives a necessary and sufficient condition for the existence of the weak solution. Moreover, a result by Kurenok \cite{kur06} states that the SDE \begin{equation*}
	dX_t = b(t,X_{t-}) \, dt + \, dL_t, \qquad X_0 \sim \delta_x
\end{equation*}
has a weak solution if $b$ is a bounded measurable function and the characteristic exponent $\psi: \mbb{R}^d \to \mbb{C}$ of the L\'evy process $(L_t)_{t \geq 0}$ satisfies \begin{equation*}
	\lim_{|\xi| \to \infty} \frac{|\xi|}{\re \psi(\xi)} = 0. 
\end{equation*}

In this section we will be derive a new existence result for martingale problems with discontinuous coefficients, cf.\ Theorem~\ref{exi-9}. This will allow us to establish a new existence result for L\'evy-driven SDEs with discontinuous coefficients, see Corollary~\ref{exi-13}. As usual we denote by $(X_t)_{t \geq 0}$ the canonical process on $\Omega := D[0,\infty)$. We start with the following, rather simple observation. 

\begin{prop} \label{exi-5}
	Let $A: C_c^{\infty}(\mbb{R}^d) \to \mc{B}_b(\mbb{R}^d)$ and $L: C_c^{\infty}(\mbb{R}^d) \to \mc{B}_b(\mbb{R}^d)$ be two linear operators such that \begin{equation*}
		Af(x)  = Lf(x) \quad \text{for Lebesgue almost all $x \in \mbb{R}^d$}
	\end{equation*}
	for all $f \in C_c^{\infty}(\mbb{R}^d)$. If $(X_t,\mc{F}_t,\mbb{P}^x; x \in \mbb{R}^d, t \geq 0)$ is a Markov process which solves the $(A,C_c^{\infty}(\mbb{R}^d))$-martingale problem and $(X_t)_{t \geq 0}$ admits a transition density $p$ with respect to Lebesgue measure, then $(X_t,\mc{F}_t,\mbb{P}^x;x \in \mbb{R}^d, t \geq 0)$ is a Markovian solution to the $(L,C_c^{\infty}(\mbb{R}^d))$-martingale problem.
\end{prop}

Roughly speaking, the process $(X_t)_{t \geq 0}$ does not ``see'' Lebesgue null sets (since it has a transition density with respect to Lebesgue measure), and therefore we can modify $Af$ on a Lebesgue null set. 

\begin{proof}[Proof of Proposition~\ref{exi-5}]
	For any $y \in \mbb{R}^d$, $f \in C_c^{\infty}(\mbb{R}^d)$ and $s \leq t$ we have \begin{align*}
		\mbb{E}^y \left( \int_s^t Af(X_r) \, dr \right)
		= \int_s^t\!\!\int_{\mbb{R}^d} Af(z) p_r(y,z) \, dz \, dr 
		&= \int_s^t\!\!\int_{\mbb{R}^d} Lf(z) p_r(y,z) \, dz \, dr \\
		&= \mbb{E}^y \left( \int_s^t Lf(X_r) \, dr \right).
	\end{align*}
	Using that $(X_t)_{t \geq 0}$ is a solution to the $(A,C_c^{\infty}(\mbb{R}^d))$-martingale problem which has the Markov property we find \begin{align*}
		0
		&=\mbb{E}^x \left( \prod_{i=1}^m g_i(X_{t_i}) \left[ f(X_t)-f(X_s) - \int_s^t Af(X_r) \right] \right) \\
		&=\mbb{E}^x \left( \prod_{i=1}^m g_i(X_{t_i}) \left[ f(X_t)-f(X_s) - \int_s^t Lf(X_r) \right] \right)
	\end{align*}
	for any $g_i \in C_b(\mbb{R}^d)$ and $0 \leq t_1 \leq \ldots \leq t_m \leq s \leq t$. This shows that $\mbb{P}^x$ is a solution to the $(L,C_c^{\infty}(\mbb{R}^d))$-martingale with initial distribution $\mu = \delta_x$.
\end{proof}

Proposition~\ref{exi-5} is a useful tool to derive existence results for the particular case that the symbol $q$ is ``nice'' up to a null set.

\begin{bsp}[Isotropic stable-like process] \label{exi-7}
	Let $\alpha: \mbb{R}^d \to (0,2]$ be a H\"older continuous mapping which is bounded away from from $0$. If $\beta: \mbb{R}^d \to (0,2]$ satisfies $\alpha=\beta$ Lebesgue-almost everywhere, then there exists a Feller process which solves the martingale problem for the pseudo-differential operator with symbol $p(x,\xi) = |\xi|^{\beta(x)}$, $x,\xi \in \mbb{R}^d$.
\end{bsp}

\begin{proof}
	It is known that there exists a Feller process with symbol $q(x,\xi) = |\xi|^{\alpha(x)}$ and that the process admits a transition density, cf.\ \cite{diss} or \cite{matters}. As \begin{equation*}
		Af(x) := - \int_{\mbb{R}^d} e^{ix \cdot \xi} q(x,\xi) \hat{f}(\xi) \, d\xi = - \int_{\mbb{R}^d} e^{ix \cdot \xi} p(x,\xi) \hat{f}(\xi) \, d\xi =: Lf(x)
	\end{equation*}
	for Lebesgue almost all $x \in \mbb{R}^d$, we have $Af=Lf$ almost everywhere; applying Proposition~\ref{exi-5} finishes the proof.
\end{proof}

A possible choice for $\beta$ is, for instance, $\beta(x) = \alpha(x) \I_{\mbb{R}^d \backslash A}$ for a Lebesgue null set $A \subseteq \mbb{R}^d$. Let us remark that Example~\ref{exi-7} works in a similar fashion for other stable-like processes, for instance relativistic stable-like processes or Lamperti stable-like processes, cf.\ \cite{matters}.  \par \medskip

The main result in this section is the following existence result. Recall that $(X_t)_{t \geq 0}$ denotes the canonical process.

\begin{thm} \label{exi-9}
	Let $A_n$, $n \geq 1$, be a pseudo-differential operator with symbol $q_n$ such that $q_n(x,0)=0$. For fixed $\mu \in \mc{P}(\mbb{R}^d)$ let $\mbb{P}_n$, $n \geq 1$, be a solution to the $(A_n,C_c^{\infty}(\mbb{R}^d))$-martingale problem with initial distribution $\mu$. Assume that the following assumptions are satisfied. 
	\begin{enumerate}[label*=\upshape (C\arabic*),ref=\upshape C\arabic*] 
		\item\label{C1} (Local equiboundedness) \begin{align*}
			\forall R>0: \quad \sup_{n \geq 1} \sup_{|x| \leq R} \left( |b_n(x)| + |Q_n(x)| + \int_{y \neq 0} \min\{|y|^2,1\} \, \nu_n(x,dy) \right) &< \infty;
		\end{align*}
		here $(b_n,Q_n,\nu_n)$ denotes the characteristics of $q_n$;
		\item\label{C2} (Uniform equicontinuity at $\xi=0$) $\lim_{R \to \infty} \sup_{n \geq 1} \sup_{|y| \leq R} \sup_{|\xi| \leq R^{-1}} |q_n(y,\xi)|=0$;
		\item\label{C3} (Krylov estimate) There exist a locally finite measure $m$ on $(\mbb{R}^d,\mc{B}(\mbb{R}^d))$ and a constant $p \geq 1$ such that for any $T>0$ \begin{equation}
			\mbb{E}_{\mbb{P}_n} \left( \int_0^t u(X_s) \, ds \right) \leq c \|u\|_{L^p(m)}, \qquad u \in \mc{B}_b(\mbb{R}^d), u \geq 0, n \in \mbb{N}, t  \in [0,T] \label{exi-eq10}
		\end{equation}
		for some absolute constant $c=c(T)>0$.
	\end{enumerate}
	If $L: C_c^{\infty}(\mbb{R}^d) \to \mc{B}_b(\mbb{R}^d)$ is a linear operator such that \begin{equation}
		\inf_{g \in C_b(\mbb{R}^d)} \left( \limsup_{n \to \infty} \|A_n f-g\|_{L^p(m)} + \|Lf-g\|_{L^p(m)} \right)=0 \fa f \in C_c^{\infty}(\mbb{R}^d), \label{exi-eq11}
	\end{equation}
	then there exists a solution to the $(L,C_c^{\infty}(\mbb{R}^d))$-martingale problem with initial distribution $\mu$. 
\end{thm}

We will construct the solution $\mbb{P}$ as the weak limit of (a subsequence of) $(\mbb{P}_n)_{n \in \mbb{N}}$; \eqref{C1} and \eqref{C2} give tightness of $(\mbb{P}_n)_{n \in \mbb{N}}$ whereas \eqref{C3} and \eqref{exi-eq11} are used to show that the weak limit $\mbb{P}$ is indeed a solution to the $(L,C_c^{\infty}(\mbb{R}^d))$-martingale problem. 

\begin{bem_thm} \label{exi-11} \begin{enumerate}
	\item\label{exi-11-ii} If $Lf = \lim_{n \to \infty} A_nf$ then \eqref{exi-eq11} is equivalent to \begin{equation*}
			\inf_{g \in C_b(\mbb{R}^d)} \|Lf-g\|_{L^p(m)} = 0, \qquad f \in C_c^{\infty}(\mbb{R}^d).
	\end{equation*} %dominated conv. theorem, cont-9(iii)
	This condition is automatically satisfied if $m$ is a finite measure; indeed, if $m$ is finite, then $C_b(\mbb{R}^d)$ is dense in $L^p(m)$ and $Lf \in \mc{B}_b(\mbb{R}^d) \subseteq L^p(m)$.
	\item\label{exi-iv} By \cite[Lemma 6.2]{schnurr}, the boundedness condition \eqref{C1} is equivalent to \begin{equation*}
		\forall R>0: \quad \sup_{n \geq 1} \sup_{|x| \leq R} \sup_{|\xi| \leq 1} |q_n(x,\xi)| < \infty.
	\end{equation*}
	\item\label{exi-11-i} Condition  \eqref{C3} implies, by the Radon-Nikod\'ym theorem, that the distribution $\mbb{P}_n(X_t \in \cdot)$ is absolutely continuous with respect to $m$ for Lebesgue almost every $t>0$. 
	\item\label{exi-11-iii} We will see in the proof of Theorem~\ref{exi-9} that the solution $\mbb{P}$ satisfies the Krylov estimate \begin{equation*}
		\mbb{E}_{\mbb{P}} \left( \int_0^t u(X_s) \, ds \right) \leq c \|u\|_{L^p(m)}, \qquad u \in \mc{B}_b(\mbb{R}^d), u \geq 0, n \in \mbb{N}, t \in [0,T].
	\end{equation*}
	In particular, $\mbb{P}(X_t \in \cdot)$ is absolutely continuous with respect to $m$ for Lebesgue almost all $t>0$.
	\item\label{exi-11-v} Inequality \eqref{exi-eq10} is automatically satisfied for functions $u \geq 0$ such that $\|u\|_{L^p(m)}=\infty$. The local finiteness of $m$ ensures that $L^p(m)$ is sufficiently rich (in particular, $C_c(\mbb{R}^d) \subseteq L^p(m)$).
\end{enumerate} \end{bem_thm}

For the proof of Theorem~\ref{exi-9} we need some auxiliary statements.

\begin{lem} \label{cont-12}
	Let $A$ be a pseudo-differential operator with symbol $q$, $q(x,0)=0$, and characteristics $(b,Q,\nu)$. If $f \in C_c^{\infty}(\mbb{R}^d)$ is such that the support of $f$ is contained in the closed ball $\overline{B(0,R)}$ for some $R>0$, then \begin{align}  \label{cont-eq11} \begin{aligned}
			\|Af\|_{\infty} &\leq 2\|f\|_{(2)} \sup_{|x| \leq R} \left( |b(x)|+|Q(x)| + \int_{y \neq 0} \min\{|y|^2,1\} \, \nu(x,dy) \right) \\ &\quad + \|f\|_{\infty} \sup_{|x|>R} \nu(x,\overline{B(-x,R)}). \end{aligned}
		\end{align}
	Moreover, there exist absolute constants $C_1,C_2>0$ (not depending on $R$ and $f$) such that 
		 \begin{align} \label{cont-eq14} \begin{aligned}
			\|Af\|_{\infty} &\leq 2\|f\|_{(2)} \sup_{|x| \leq R} \left( |b(x)|+|Q(x)| + \int_{y \neq 0} \min\{|y|^2,1\} \, \nu(x,dy) \right) \\ &\quad + C_2 \|f\|_{\infty} \sup_{|x|>R} \sup_{|\xi| \leq |x|^{-1}} |\re q(y,\xi)| \end{aligned}
		\end{align}	
		and 
			 \begin{align}
					\|Af\|_{\infty} &\leq C_1 \|f\|_{(2)}\sup_{|x| \leq R} \sup_{|\xi| \leq 1} |q(y,\xi)|  + C_2 \|f\|_{\infty} \sup_{|x|>R} \sup_{|\xi| \leq |x|^{-1}} |\re q(y,\xi)|. \label{cont-eq16}
				\end{align}	
\end{lem}

\begin{proof}
	 Fix $f \in C_c^{\infty}(\mbb{R}^d)$ and $R>0$ such that $\spt f \subseteq \overline{B(0,R)}$. If $|x| \leq R$ then by Taylor's formula \begin{equation*}
			|A f(x)| \leq 2 \|f\|_{(2)} \sup_{|x| \leq R} \left( |b(x)| +|Q(x)| + \int_{y \neq 0} \min\{1,|y|^2\} \, \nu(x,dy) \right).
			\end{equation*}
			On the other hand, we have for $|x|>R$ \begin{equation*}
			|A f(x)| = \left| \int_{y \neq 0} f(x+y) \, \nu(x,dy) \right| \leq \|f\|_{\infty} \nu(x,\overline{B(-x,R)}),
			\end{equation*}
			and combining the estimates gives \eqref{cont-eq11}.  Since \begin{equation*}
					\nu(x,\overline{B(-x,R)}) \leq C_2 \sup_{|\xi| \leq |x|^{-1}} |\re q(x,\xi)|
				\end{equation*}
			for some absolute constant $C_2>0$, see e.\,g.\ \cite[Proof of Theorem 1.27]{diss} or \cite[Proof of Lemma 3.26]{ltp}, we get \eqref{cont-eq14}. Finally, \eqref{cont-eq16} follows from \cite[Lemma 6.2]{schnurr} and \eqref{cont-eq14}.
\end{proof}

The following maximal inequality is a crucial tool for the proof of Theorem~\ref{exi-9} but also for the proof of the Markovian selection theorem in Section~\ref{cont}.

\begin{prop} \label{cont-9}
	Let $A: C_c^{\infty}(\mbb{R}^d) \to \mc{B}_b(\mbb{R}^d)$ be a pseudo-differential operator with symbol $q$, $q(x,0)=0$, and let $\mbb{P}^{\mu}$ be a solution to the $(A,C_c^{\infty}(\mbb{R}^d))$-martingale problem with initial distribution $\mu$. Then there exists an absolute constant $c>0$ (not depending on $\mu$ or $\mbb{P}^{\mu}$) such that \begin{equation*}
		\mbb{P}^{\mu} \left( \sup_{s \leq t} |X_s| \geq R, |X_0| \leq r \right)
		\leq c t \sup_{|y| \leq R}\sup_{|\xi| \leq R^{-1}} |q(y,\xi)|
	\end{equation*}
	for any $t>0$ and $R \geq 2r>0$.
\end{prop} 

For Feller processes $(X_t)_{t \geq 0}$ the maximal inequality goes back to Schilling \cite{rs} (see also \cite[Theorem 5.1]{ltp}), and has been refined in \cite{matters}. A localized maximal inequality was derived in \cite{ihke}, and \cite{timechange} gives a maximal inequality for solutions to martingale problems.

\begin{proof}[Proof of Proposition~\ref{cont-9}]
	The reasoning is similar to the proof of \cite[Theorem 5.1]{ltp} (see also \cite[Lemma 3.1]{timechange}) but for the readers' convenience we sketch the idea of the proof.  Fix $0 < r \leq 2R < \infty$ and $u \in C_c^{\infty}(\mbb{R}^d)$, $0 \leq u \leq 1$ such that $u|_{B(0,1/2)}=1$ and $u|_{B(0,1)^c}=0$. If we set $u_R := u(\cdot/R)$ and \begin{equation*}
			\tau_R := \inf\{t>0; X_t \notin B(0,R)\}
		\end{equation*}
		then it follows from the optional stopping theorem that \begin{equation*}
			M_t := u_R(X_{t \wedge \tau_R})-u_R(X_0)- \int_0^{t \wedge \tau_R} Au_R(X_s) \, ds
		\end{equation*}
		is a $\mbb{P}^{\mu}$-martingale; in particular, \begin{equation}
			\mbb{E}_{\mbb{P}^{\mu}}\big(u_R(X_0)-u_R(X_{t \wedge \tau_R}) \big)= - \mbb{E}_{\mbb{P}^{\mu}} \left( \int_0^{t \wedge \tau_R} Au_R(X_s) \, ds \right). \label{cont-eq23}
		\end{equation}
		For any $\omega \in \{\tau_R \leq t\} \cap \{|X_0| \leq r\}$ we have $|X_{t \wedge \tau_R}(\omega)| \geq R$ and $|X_0(\omega)| \leq r \leq R/2$; thus \begin{equation*}
			u_R(X_0(\omega))-u_R(X_{t \wedge \tau_R}(\omega)) = 1
		\end{equation*}
		which implies \begin{equation*}
			\mbb{P}^{\mu} \left(\sup_{s \leq t} |X_s| \geq R, |X_0| \leq r \right) \leq \mbb{E}_{\mbb{P}^{\mu}} \big(u_R(X_0)- u_R(X_{t \wedge \tau_R}) \big).
		\end{equation*}
		Using \eqref{cont-eq23} and exactly the same reasoning as in \cite[Proof of Theorem 5.1]{ltp}, we get \begin{align*}
			\mbb{P}^{\mu} \left(\sup_{s \leq t} |X_s| \geq R, |X_0| \leq r \right)
			&\leq - \mbb{E}_{\mbb{P}^{\mu}} \left( \int_0^{t \wedge \tau_R} Au_R(X_s) \, ds \right) \\
			&= \mbb{E}_{\mbb{P}^{\mu}} \left[\int_0^{t \wedge \tau_R} \left( \I_{|y|  < R} \int_{\mbb{R}^d} e^{iy \cdot \xi} q(y,\xi) \hat{u}_R(\xi) \, d\xi \right) \bigg|_{y=X_{s-}} \, ds \right] \\
			&\leq ct \sup_{|y| \leq R} \sup_{|\xi| \leq R^{-1}} |q(y,\xi)|
		\end{align*}
		where $c := 2 \int_{\mbb{R}^d} (1+|\eta|^2) |\hat{u}(\eta)| \, d\eta$.
\end{proof}

From Proposition~\ref{cont-9} we can deduce the following statement on the tightness of a sequence of solutions to martingale problems.

\begin{kor}  \label{exi-8}
	For  $k \geq 1$ let $(q_k(x,\cdot))_{x \in \mbb{R}^d}$ be a family of continuous negative definite mappings with characteristics $(b_k,Q_k,\nu_k)$ such that $q_k(x,0)=0$. Let $\mbb{P}_k$ be a solution to the $(-q_k(x,D),C_c^{\infty}(\mbb{R}^d))$-martingale problem with initial distribution $\mu$. If \begin{align} \label{cont-eq22}
			\forall R>0: \quad \sup_{k \geq 1} \sup_{|x| \leq R} \left( |b_k(x)| + |Q_k(x)| + \int_{y \neq 0} \min\{|y|^2,1\} \,\nu_k(x,dy)\right) &< \infty \end{align}
		and \begin{align}
			\lim_{R \to \infty} \sup_{k \geq 1} \sup_{|y| \leq R} \sup_{|\xi| \leq R^{-1}} |q_k(y,\xi)| = 0, \label{cont-eq25}
		\end{align}
		then $(\mbb{P}_k)_{k \geq 1}$ is tight.
\end{kor}

\begin{proof}
	For fixed $\eps>0$ there exists $r>0$ such that $\mu(B(0,r)^c) \leq \eps$. Applying Proposition~\ref{cont-9} we find \begin{align*}
		\mbb{P}_k \left( \sup_{t \leq T} |X_t| \geq R \right)
		&\leq \eps + \mbb{P}_k \left( \sup_{t \leq T} |X_t| \geq R, |X_0| \leq r \right)
		\leq \eps + cT \sup_{|y| \leq R} \sup_{|\xi| \leq R^{-1}} |q_k(y,\xi)|
	\end{align*}
	for some absolute constant $c>0$. By \eqref{cont-eq25} this implies that the compact containment condition \begin{equation*}
			\lim_{R \to \infty} \sup_{k \geq 1} \mbb{P}_k \left( \sup_{t \leq T} |X_t| \geq R \right)=0
		\end{equation*}
		holds for any $T>0$. Moreover, it is not difficult to see that \eqref{cont-eq25} gives \begin{equation*}
			\sup_{k \geq 1} \sup_{|x|>R} \sup_{|\xi| \leq |x|^{-1}} |q_k(x,\xi)| < \infty,
		\end{equation*}
		and therefore we find from Lemma~\ref{cont-12} that $\sup_{k \geq 1} \|-q_k(x,D) f\|_{\infty} < \infty$. Now the assertion follows from Aldous tightness condition, cf.\ \cite[Theorem 4.1.16]{jac3}.
	\end{proof}

We are now ready to prove Theorem~\ref{exi-9}.

\begin{proof}[Proof of Theorem~\ref{exi-9}]
	It follows from \eqref{C1},\eqref{C2} and Corollary~\ref{exi-8} that the sequence $(\mbb{P}_n)_{n \geq 1}$ is tight, and therefore the weak limit $\mbb{P} = \lim_{k \to \infty} \mbb{P}_{n_k}$ exists for a suitable subsequence. It remains to prove that $\mbb{P}$ is a solution to the $(L,C_c^{\infty}(\mbb{R}^d))$-martingale problem. We claim that $\mbb{P}$ satisfies the Krylov estimate \begin{equation}
		\mbb{E}_{\mbb{P}} \left( \int_0^t u(X_s) \, ds \right) \leq c \|u\|_{L^p(m)} \fa u \in \mc{B}_b(\mbb{R}^d), u \geq 0, t \in [0,T]. \label{exi-eq15}
	\end{equation}
	Indeed: If $u=\I_A$ for some open set $A \subseteq \mbb{R}^d$, then this is a direct consequence of the Portmanteau theorem, Fatous lemma and \eqref{C3}; for general $u \geq 0$ the Krylov estimate then follows from a straight-forward application of the monotone class theorem. \par 
	In order to show that $\mbb{P}$ is a solution to the $(L,C_c^{\infty}(\mbb{R}^d))$-martingale problem, it suffices to prove that \begin{equation}
	\Delta := \mbb{E} \left[ \prod_{i=1}^N g_i(X_{t_i}) \left( f(X_t)-f(X_s)- \int_s^t Lf(X_r) \, dr \right) \right]=0 \tag{$\star$} \label{eq-star}
	\end{equation}
	for any $0 \leq t_1 \leq \ldots \leq t_N \leq s \leq t$, $f \in C_c^{\infty}(\mbb{R}^d)$ and $g_i \in C_b(\mbb{R}^d)$, $0 \leq g_i \leq 1$. Fix $\eps>0$. By \eqref{exi-eq11}, we can choose $g \in C_b(\mbb{R}^d)$ such that \begin{equation}
		\limsup_{k \to \infty} \|A_{n_k} f-g\|_{L^p(m)} + \|Lf-g\|_{L^p(m)} \leq \eps. \label{exi-eq19}
	\end{equation}
	Writing $Lf = (Lf-g)+g$ in \eqref{eq-star} we get $\Delta = \Delta_1+\Delta_2$ where \begin{align*}
		\Delta_1 
		&:= \mbb{E}_{\mbb{P}} \left[ \prod_{i=1}^N g_i(X_{t_i}) \left( f(X_t)-f(X_s)- \int_s^t g(X_r) \, dr \right) \right] \\
		\Delta_2 &:= \mbb{E}_{\mbb{P}} \left[ \prod_{i=1}^N g_i(X_{t_i}) \left( \int_0^t (Lf-g)(X_r) \, dr \right) \right].
	\end{align*}
	We estimate the terms separately. It follows from \eqref{exi-eq15} and \eqref{exi-eq19} that \begin{equation*}
		|\Delta_2| \leq c \|Lf-g\|_{L^p(m)}  \prod_{i=1}^n \|g_i\|_{\infty} \leq c \eps.
	\end{equation*}
	Since $g$ is continuous, the weak convergence of $\mbb{P}_{n_k}$ to $\mbb{P}$ gives \begin{align*}
		\Delta_1
		&= \lim_{k \to \infty} \mbb{E}_{\mbb{P}_{n_k}} \left[ \prod_{i=1}^N g_i(X_{t_i}) \left( f(X_t)-f(X_s) - \int_s^t g(X_r) \, dr \right) \right].
	\end{align*}
	Using that $\mbb{P}_{n_k}$ solves the $(A_{n_k},C_c^{\infty}(\mbb{R}^d))$-martingale problem we obtain \begin{equation*}
		\Delta_1 = \lim_{k \to \infty} \mbb{E}_{\mbb{P}_{n_k}} \left[ \prod_{i=1}^N g_i(X_{t_i}) \int_s^t (g-A_{n_k}f)(X_r) \, dr \right].
	\end{equation*}
	Thus, by \eqref{C3} and \eqref{exi-eq19}, \begin{equation*}
		|\Delta_1| \leq c \limsup_{k \to \infty} \|g-A_{n_k} f\|_{L^p(m)} \leq c \eps. \qedhere
	\end{equation*}
\end{proof}

Let us illustrate Theorem~\ref{exi-9} with some examples. We obtain the following existence result for solutions to SDEs with not necessarily continuous coefficients.

\begin{kor} \label{exi-13}
	Let $(L_t)_{t \geq 0}$ be a one-dimensional L\'evy process with characteristic exponent $\psi$ satisfying the following assumptions. 
		\begin{enumerate}[label*=\upshape (L\arabic*),ref=\upshape L\arabic*] 
			\item\label{L1} $\psi$ has a holomorphic extension $\Psi$ to the domain \begin{equation*}
				U := U(\vartheta) := \{z \in \mbb{C} \backslash \{0\}; (\arg z) \mod \pi \in (-\vartheta,\vartheta)\}
			\end{equation*}
			for some $\vartheta \in (0,\pi/2)$; here $\arg z \in (-\pi,\pi]$ denotes the argument of $z \in \mbb{C}$.
			\item\label{L2} There exist constants $\alpha, \beta \in (0,2]$ and $c_1,c_2>0$ such that \begin{equation*}
				\re \Psi(z) \geq c_1 |\re z|^{\beta}, \qquad z \in U, |z| \gg 1,
			\end{equation*}
			and \begin{equation*}
				|\Psi(z)| \leq c_2 |z|^{\alpha} \I_{\{|z| \leq 1\}} + c_2 |z|^{\beta} \I_{\{|z|>1\}}, \qquad z \in U.
			\end{equation*}
			\item\label{L3} $|\Psi'(z)| \leq c_2 |z|^{\beta-1}$ for all $|z| \gg 1$, $z \in U$.
		\end{enumerate}
		Let $b: \mbb{R} \to \mbb{R}$ and $\sigma: \mbb{R} \to (0,\infty)$ be bounded measurable functions. If \begin{equation}
			\beta>1 \quad \text{or} \quad b=0 \label{exi-eq29}
		\end{equation}
		and \begin{equation*}
			 \inf_{x \in \mbb{R}} \sigma(x)>0,
		\end{equation*}
		then there exists for any $\mu \in \mc{P}(\mbb{R}^d)$ a weak solution to the L\'evy-driven SDE \begin{equation}
			dX_t = b(X_{t-}) \, dt + \sigma(X_{t-}) \, dL_t, \qquad X_0 \sim \mu \label{exi-eq31}
		\end{equation}
		For Lebesgue-almost every $t>0$ the distribution $\mbb{P}^{\mu}(X_t \in \cdot)$ is absolutely continuous with respect to Lebesgue measure. 
	\end{kor}

Corollary~\ref{exi-13} applies, for instance, if $(L_t)_{t \geq 0}$ is isotropic stable, relativistic stable, Lamperti stable or a truncated L\'evy process; see \cite[Table 5.2]{matters} for further examples of L\'evy processes satisfying \eqref{L1}-\eqref{L3}. Corollary~\ref{exi-13} generalizes, in particular, \cite[Theorem 4.1]{kur08}  which is restricted to isotropic stable driving L\'evy processes. Let us remark that \eqref{exi-eq29} means that the jump part dominates the drift part.

\begin{proof}[Proof of Corollary~\ref{exi-13}]
	Let $(L_t)_{t \geq 0}$ be a L\'evy process satisfying \eqref{L1}-\eqref{L3}. We split the proof in two parts; in the first part we will derive a Krylov estimate for SDEs with H\"{o}lder continuous coefficients, and in the second part we will approximate the coefficients $b,\sigma$ by H\"{o}lder continuous functions in order to apply Theorem~\ref{exi-9}. \par
	\textbf{Step 1:} Let $f,g$ be bounded H\"{o}lder continuous functions such that $\inf_x g(x)>0$. In \cite{matters} (see also \cite{diss}) it was shown that there exists a Feller process $(X_t,\mc{F}_t,\mbb{P}^x; x \in \mbb{R}^d, t \geq 0)$ which is the unique weak solution to the SDE \begin{equation}
		dX_t = f(X_{t-}) \, dt + g(X_{t-}) \, dL_t. \label{exi-eq335}
	\end{equation}
	The Feller process $(X_t)_{t \geq 0}$ has a continuous transition probability $p_t(x,y)$. Using the heat kernel estimates from \cite{matters} we find that there exists a continuous function $C$ such that \begin{equation}
		\int_0^t p_s(x,y) \, ds \leq C(T,\|f\|_{\varrho(f)}, \|g\|_{\varrho(f)}, 1/\inf_x g(x)) Q(x-y), \qquad x,y \in \mbb{R}, \, \, t \in (0,T] \label{exi-eq34}
	\end{equation}
	where $\varrho(f)$ and $\varrho(g)$ denote the H\"{o}lder exponent of $f$ and $g$, respectively, and \begin{equation*}
		Q(z) := |z|^{-1-\alpha \wedge \beta} \I_{|z| \geq 1} + (|z|^{-1+\beta} + |\log|z||) \I_{0<|z| \leq 1} + \I_{|z|=0},
	\end{equation*}
	see the appendix for details. Since the transition probability $p$ is continuous, it is not difficult to see that $x \mapsto \mbb{P}^x(A)$ is measurable for any $A \in \mc{F}_{\infty}$, and therefore $\mbb{P}^{\mu} := \int \mbb{P}^x \, \mu(dx)$ defines a probability measure; it is a weak solution to \eqref{exi-eq335} with initial distribution $\mu \in \mc{P}(\mbb{R}^d)$, and \begin{equation*}
		\mbb{P}^{\mu}(X_s \in B) 
		= \int_{\mbb{R}^d} \mbb{P}^x(X_s \in B) \, \mu(dx)
		= \int_{\mbb{R}^d} \int_B p_s(x,y) \, dy \, \mu(dx)
	\end{equation*}
	for any $B \in \mc{B}(\mbb{R}^d)$, $s>0$. Consequently, we obtain from Fubini's theorem and \eqref{exi-eq34}
	 \begin{align}
		 \int_0^t \mbb{E}_{\mbb{P}^{\mu}} u(X_s) \, ds
		&= \int_0^t \int_{\mbb{R}^d} \int_{\mbb{R}^d} u(y) p_s(x,y) \, dy \, \mu(dx) \, ds  \notag \\
		&\leq C(T,\|f\|_{\varrho(f)}, \|g\|_{\varrho(f)}, 1/\inf_x g(x)) \int u(y) \, \left(\int_{\mbb{R}^d} Q(x-y) \, \mu(dx) \right) \, dy \label{exi-eq35}
	\end{align}
	for any function $u \geq 0$, $u \in \mc{B}_b(\mbb{R}^d)$, i.\,e.\ a Krylov estimate holds for $p=1$ and the measure \begin{equation*}
		m(dy) := \left(\int_{\mbb{R}^d} Q(x-y) \, \mu(dx) \right) \, dy.
	\end{equation*}
	Note that $m$ is a finite measure since, by Tonelli's theorem and the invariance of Lebesgue measure under translations, \begin{equation*}
		m(\mbb{R}^d) = \int_{\mbb{R}^d} \int_{\mbb{R}^d} Q(z) \, dz \, \mu(dx) = \int_{\mbb{R}^d} Q(z) \, dz < \infty.
	\end{equation*}
	\textbf{Step 2:} Let $b$ and $\sigma$ be as in Corollary~\ref{exi-13}. By Lemma~\ref{app-1}, we can choose sequences $(f_n)_{n \in \mbb{N}}$, $(g_n)_{n \in \mbb{N}} \subseteq C(\mbb{R})$ and $(\alpha_n)_{n \in \mbb{N}}$, $(\beta_n)_{n \in \mbb{N}} \subseteq (0,1]$ such that \begin{align} \label{exi-eq33} \begin{aligned}
		\sup_{n \in \mbb{N}} \|f_n\|_{\alpha_n} + \sup_{n \in \mbb{N}} \|g_n\|_{\beta_n} &< \infty \\
		\inf_{x \in \mbb{R}} \sigma(x) \leq g_n(x) \leq \|\sigma\|_{\infty} \qquad \|f_n\|_{\infty} &\leq \|b\|_{\infty}.
	\end{aligned} \end{align}
	and \begin{equation*}
		f_n(x) \xrightarrow[]{n \to \infty} b(x) \qquad g_n(x) \xrightarrow[]{n \to \infty} \sigma(x)
	\end{equation*}
	Lebesgue almost everywhere. If we denote by $A_n$ the pseudo-differential operator with symbol $q_n(x,\xi) := - if_n(x) \xi + \psi(g_n(x)\xi)$, then Step 1 shows that there exists for each $n \in \mbb{N}$ a solution $\mbb{P}_n$ to the $(A_n,C_c^{\infty}(\mbb{R}^d))$ martingale problem with initial distribution $\mu$ which satisfies the Krylov estimate \begin{equation*}
		\int_0^t \mbb{E}_{\mbb{P}_n} u(X_s) \, ds \leq C(T,\|f_n\|_{\alpha_n}, \|g_n\|_{\beta_n}, 1/\inf_x g_n(x)) \|u\|_{L^1(m)}
	\end{equation*}
	for $C$ and $m$ defined in Step 1. Because of \eqref{exi-eq33} and the continuity of $C$ we can choose a constant $K=K(T)>0$ such that \begin{equation*}
		\int_0^t \mbb{E}_{\mbb{P}_n} u(X_s) \, ds \leq K  \|u\|_{L^1(m)} \fa n \in \mbb{N}, u \in \mc{B}_b(\mbb{R}^d), u \geq 0, t \in [0,T].
	\end{equation*}
	This shows that \eqref{C3} in Theorem~\ref{exi-9} holds for $p=1$ and the finite measure $m$.  Moreover, it can be easily verified that \eqref{exi-eq33} gives \eqref{C1}, \eqref{C2}. If we denote by $L$ the pseudo-differential operator with symbol $q(x,\xi) := -i b(x) \xi  + \psi(\sigma(x) \xi)$, then $q_n(x,\xi) \to q(x,\xi)$ for Lebesgue-almost all $x \in \mbb{R}$, and so \begin{equation*}
		\limsup_{n \to \infty} \|A_n f - g\|_{L^1(m)} = \|Lf-g\|_{L^1(m)} \fa f  \in C_c^{\infty}(\mbb{R}), g \in C_b(\mbb{R}).
	\end{equation*}
	Since $m$ is a finite measure, we know that $C_b(\mbb{R})$ is dense in $L^1(m)$, and as $Lf \in \mc{B}_b(\mbb{R}) \subseteq L^1(m)$ this implies \begin{equation*}
		\inf_{g \in C_b(\mbb{R})} \|Lf-g\|_{L^1(m)}=0.
	\end{equation*}
	Applying Theorem~\ref{exi-9} we find that there exists a solution to the $(L,C_c^{\infty}(\mbb{R}))$-martingale problem with initial distribution $\mu$. It is known that the solution is a weak solution to \eqref{exi-eq31}, see \cite{kurtz}. The absolute continuity of the distribution follows from Remark~\ref{exi-11}\eqref{exi-11-i}.
\end{proof}

Using the heat kernel estimates in \cite[Section 5.3]{matters} and an approximation procedure as in the proof of Corollary~\ref{exi-13} we can use Theorem~\ref{exi-9} to derive results on mixed processes and stable-like processes.

\begin{bsp}[Mixed L\'evy processes] \label{exi-17} 
	Let $\psi_1,\psi_2: \mbb{R} \to \mbb{R}$ be two continuous negative definite functions satisfying \eqref{L1}-\eqref{L3} from Corollary~\ref{exi-13}. For two measurable bounded mappings $\varphi_1,\varphi_2: \mbb{R} \to (0,\infty)$  we denote by $A$ the pseudo-differential operator with symbol \begin{equation}
		q(x,\xi) := \varphi_1(\xi) \psi_1(\xi) + \varphi_2(x) \psi_2(\xi),  \qquad x, \xi \in \mbb{R}. \label{exi-eq37}
	\end{equation}
	If \begin{equation*}
		\inf_{x \in \mbb{R}} (\varphi_1(x)+\varphi_2(x))>0
	\end{equation*}
	then there exists for any $\mu \in \mc{P}(\mbb{R}^d)$ a solution to the $(A,C_c^{\infty}(\mbb{R}))$-martingale problem with initial distribution $\mu$.
\end{bsp}

Example~\ref{exi-17} applies, for instance, if $\psi_1(\xi) = |\xi|^{\gamma}$ (isotropic stable) and $\psi_2(\xi) = \sqrt{|\xi|^{2}+m^2}^{\varrho/2} - m^{\varrho}$ (relativistic stable) for some $\gamma,\varrho \in (0,2]$ and $m>0$; we refer to \cite[Table 5.2]{matters} for further examples of continuous negative definite functions satisfying \eqref{L1}-\eqref{L3}. We would like to remark that Example~\ref{exi-17} can be extended to higher dimensions; for $d>1$ we have to replace \eqref{L1} by the assumption that $\psi_i(\xi) = \Psi_i(|\xi|)$, $\xi \in \mbb{R}^d$, for a function $\Psi_i$ which is holomorphic on $U$ (defined in \eqref{L1}) and which satisfies the growth conditions \eqref{L2},\eqref{L3}. Let us mention that the existence of (Feller) processes with a decomposable symbol of the form \eqref{exi-eq37} has been studied in \cite{hoh94,kol04} (for smooth $\varphi_i$) and in \cite[Theorem 5.5]{timechange} (for continuous $\varphi_i$). 

\begin{bsp}[Stable-like processes] \label{exi-19}
	Let $I = [\alpha_0,\alpha_1] \subseteq (0,2)$, $I \neq \emptyset$, and $J \subseteq \mbb{R}^n$ be an open set. Let $f: I \times J \to (0,\infty)$ be a bounded function such that \begin{enumerate}
		\item $\beta \mapsto f(\alpha,\beta)$ is differentiable for each $\alpha \in I$ and $\sup_{(\alpha,\beta) \in I \times J} |\partial_{\beta_j} f(\alpha,\beta)|< \infty$ for all $j \in \{1,\ldots,n\}$,
		\item $f_0 := \inf_{(\alpha,\beta) \in I \times J} f(\alpha,\beta)>0$.
	\end{enumerate}
	For a Borel measurable function $\varphi: \mbb{R}^d \to J$ denote by $A$ the pseudo-differential operator with symbol \begin{equation}
		q(x,\xi) := \int_I |\xi|^{\alpha} f(\alpha,\varphi(x)) \, d\alpha, \qquad x,\xi \in \mbb{R}^d. \label{exi-eq39}
	\end{equation}
	Then there exists for any $\mu \in \mc{P}(\mbb{R}^d)$ a solution to the $(A,C_c^{\infty}(\mbb{R}^d))$-martingale problem with initial distribution $\mu$.
\end{bsp}

\begin{bem}
It follows from the well-known identity \begin{equation*}
	|\xi|^{\alpha} = c_{\alpha,d} \int_{\mbb{R}^d} (1-\cos(y \cdot \xi)) \frac{1}{|y|^{d+\alpha}} \, dy, \qquad \xi \in \mbb{R}^d, \, \, \alpha \in (0,2)
\end{equation*}
that we can write the symbol \eqref{exi-eq39} in the form \begin{equation*}
	q(x,\xi) = \int_{\mbb{R}^d} (1-\cos(y \cdot \xi)) \nu(x,dy)
\end{equation*}
where \begin{equation*}
	\nu(x,dy) := c_{\alpha,d} \int_I f(\alpha,\varphi(x)) \frac{1}{|y|^{d+\alpha}} \, d\alpha \, dy.
\end{equation*} \end{bem}

\section{Markovian solutions to martingale problems for L\'evy-type operators} \label{cont}

Throughout this section, we denote by $(X_t)_{t \geq 0}$ the canonical process on $\Omega = D[0,\infty)$, and $(q(x,\cdot))_{x \in \mbb{R}^d}$ is a family of continuous negative definite functions such that $q(x,0)=0$. \par
The aim of this section is to establish a condition which ensures the existence of a Markovian solution to the $(-q(x,D),C_c^{\infty}(\mbb{R}^d))$-martingale problem. It is well-known, see e.\,g.\ \cite{ethier}, that the Markov property holds if the martingale problem is well-posed. It is, however, in general hard to verify the well-posedness of a martingale-problem. Our main result in this section, Theorem~\ref{cont-1}, states that a Markovian selection exists if the symbol $q$ satisfies a certain continuity condition at $\xi=0$.

\begin{thm}[Markovian selection theorem] \label{cont-1} 
	Let $A$ be a pseudo-differential operator with symbol $q$, $q(x,0)=0$, such that for any  $\mu \in \mc{P}(\mbb{R}^d)$ there exists a solution to the $(A,C_c^{\infty}(\mbb{R}^d))$-martingale problem with initial distribution $\mu$. If $q$ is locally bounded and satisfies \begin{equation}
		\lim_{R \to \infty} \sup_{|y| \leq R} \sup_{|\xi| \leq R^{-1}} |q(y,\xi)| = 0, \label{cont-eq3}
	\end{equation}
	then there exists a strongly Markovian solution $(X_t,\mc{F}_t,\mbb{P}^x; x \in \mbb{R}^d, t \geq 0)$ to the $(A,C_c^{\infty}(\mbb{R}^d))$-martingale problem, i.\,e.\ there exists a family of probability measures $(\mbb{P}^x)_{x \in \mbb{R}^d}$ on $D[0,\infty)$ such that \begin{enumerate}[label={\upshape(\roman*)}]
			\item\label{cont-1-i}  $(X_t,\mc{F}_t,\mbb{P}^x, x \in \mbb{R}^d, t \geq 0)$ is a conservative strong Markov process,
			\item\label{cont-1-ii} $x \mapsto \mbb{P}^x$ is measurable,
			\item\label{cont-1-iii} For any $\mu \in \mc{P}(\mbb{R}^d)$ the probability measure \begin{equation*}
				\mbb{P}^{\mu} := \int_{\mbb{R}^d} \mbb{P}^x \, \mu(dx)
			\end{equation*}
			is a solution to the $(A,C_c^{\infty}(\mbb{R}^d))$-martingale problem with initial distribution $\mu$. 
	\end{enumerate}
	Moreover, the following statement holds true:\begin{enumerate} \setcounter{enumi}{3}
		\item\label{cont-1-iv} For any fixed $f \in C_{\infty}(\mbb{R}^d)$, $f \geq 0$, and $\lambda>0$, the family $(\mbb{P}^x)_{x \in \mbb{R}^d}$ can be chosen in such a way that \begin{equation}
	\mbb{E}_{\mbb{P}^x} \left( \int_{(0,\infty)} e^{-\lambda t} f(X_t) \, dt \right) = \sup_{\mbb{P} \in \Pi_x} \mbb{E}_{\mbb{P}} \left( \int_{(0,\infty)} e^{-\lambda t} f(X_t) \, dt \right), \qquad x \in \mbb{R}^d, \label{cont-eq4}
	\end{equation}
	where $\Pi_x$ is the family of all probability measures $\mbb{P}$ solving the $(A,C_c^{\infty}(\mbb{R}^d))$-martingale problem with initial distribution $\delta_x$. \end{enumerate}
\end{thm}

If $q$ has continuous coefficients, then the assumption on the existence of a solution is automatically satisfied (cf.\ Corollary~\ref{cont-2}). For symbols $q$ with discontinuous coefficients we refer to Section~\ref{exi} for sufficient conditions ensuring the existence. \par \medskip

We will see in Section~\ref{vis} that the representation \eqref{cont-eq4} is useful in order to study properties of the function \begin{equation*}
	u(x) := \sup_{\mbb{P} \in \Pi_x} \mbb{E}_{\mbb{P}} \left( \int_{(0,\infty)} e^{-\lambda t} f(X_t) \, dt \right)
\end{equation*}
which can be understood as the resolvent with respect to the sublinear expectation $\mbb{E}_{\mbb{Q}^x} := \sup_{\mbb{P} \in \Pi_x} \mbb{E}_{\mbb{P}}$. 

\begin{kor}[Markovian selection for symbols with continuous coefficients] \label{cont-2}
	Let $A$ be a pseudo-differential operator with symbol $q$, $q(x,0)=0$. If $q$ is locally bounded, has continuous coefficients and \begin{equation*}
		\lim_{R \to \infty} \sup_{|y| \leq R} \sup_{|\xi| \leq R^{-1}} |q(y,\xi)| = 0, 
	\end{equation*}
	then the $(A,C_c^{\infty}(\mbb{R}^d))$-martingale problem admits a strongly Markovian solution $(X_t,\mc{F}_t,\mbb{P}^x, x \in \mbb{R}^d, t \geq 0)$  satisfying \ref{cont-1}.\eqref{cont-1-i}-\eqref{cont-1-iv}.
\end{kor}

We will first prove Theorem~\ref{cont-1} and Corollary~\ref{cont-2}, and then we will present some examples illustrating both results. The following result is compiled from Ethier \& Kurtz \cite{ethier}; it is the key tool for the proof of Theorem~\ref{cont-1}.

\begin{thm} \label{pre-5}
	Let $A: C_c^{\infty}(\mbb{R}^d) \to \mc{B}_b(\mbb{R}^d)$ be a linear operator, and denote by $\Pi_{\mu}$ the family of solutions to the $(A,C_c^{\infty}(\mbb{R}^d))$-martingale problem with initial distribution $\mu$. If $\Pi_{\mu} \neq \emptyset$ for any initial distribution $\mu$ and if the compact containment condition \begin{equation}
		\forall r>0, \epsilon>0, t>0 \, \, \exists R>0 \, \, \forall \mbb{P} \in \bigcup_{\mu} \Pi_{\mu}: \quad \mbb{P} \left(\sup_{s \leq t} |X_s|>R, |X_0| \leq r \right) \leq \epsilon \label{eq-cpt}
	\end{equation}
	holds, then there exist $\mbb{P}^x \in \Pi_x := \Pi_{\delta_x}$, $x \in \mbb{R}^d$, such that $(X_t,\mc{F}_t, \mbb{P}^x; x \in \mbb{R}^d, t \geq 0)$ is a strong Markov process and $x \mapsto \mbb{P}^x$ is measurable. For any fixed $f \in C_b(\mbb{R}^d)$, $f \geq 0$, and $\lambda>0$ the Markovian selection $(\mbb{P}^x)_{x \in \mbb{R}^d}$ can be chosen in such a way that \begin{equation}
		\mbb{E}_{\mbb{P}^x} \left( \int_{(0,\infty)} e^{-\lambda t} f(X_t) \, dt \right) = \sup_{\mbb{P} \in \Pi_x} \mbb{E}_{\mbb{P}} \left( \int_{(0,\infty)} e^{-\lambda t} f(X_t) \, dt \right) \fa x \in \mbb{R}^d. \label{pre-eq11}
	\end{equation}
\end{thm}

\begin{proof} 
	The first statement follows from Theorem 4.5.19 and (the proof of) Lemma 4.5.11(b) in \cite{ethier}. Let us remark that Ethier \& Kurtz assume in Lemma 4.5.11(b) that $Af$ is continuous; a close look at the proof shows, however, that this condition is not needed. The existence of a Markovian selection satisfying \eqref{pre-eq11} is a direct consequence of the proof of Theorem 4.5.19, choose $f_1 := f$ in the proof of Theorem 4.5.19.
\end{proof}

\begin{proof}[Proof of Theorem~\ref{cont-1}]
	Since $q$ is locally bounded and satisfies \eqref{cont-eq3}, Lemma~\ref{cont-12} shows that $Af$ is bounded for any $f \in C_c^{\infty}(\mbb{R}^d)$. Moreover,  it follows from \eqref{cont-eq3} and Proposition~\ref{cont-9} that the compact containment condition \eqref{eq-cpt} is satisfied. Applying Theorem~\ref{pre-5} finishes the proof.
\end{proof}

\begin{proof}[Proof of Corollary~\ref{cont-2}]
	The assertion is a direct consequence of Theorem~\ref{exi-1} and Theorem~\ref{cont-1}.
\end{proof}

Let us illustrate the Markovian selection theorems with some examples. Since there is a close connection between weak solutions to L\'evy-driven SDEs and martingale problems, Theorem~\ref{cont-1} allows us to deduce the following statement.

\begin{kor}[Markovian selection for L\'evy-driven SDEs] \label{cont-3} 
	Let $(L_t)_{t \geq 0}$ be a $d$-dimensional L\'evy process with characteristic exponent $\psi$. If $b, \sigma: \mathbb{R}^k \to \mathbb{R}^{k \times d}$ are functions of sublinear growth such that the L\'evy-driven SDE\begin{equation}
		dX_t = b(X_{t-}) \, dt + \sigma(X_{t-}) \, dL_t, \qquad X_0 \sim \mu, \label{cont-eq55}
	\end{equation}
	has a weak solution for any $\mu \in \mc{P}(\mbb{R}^d)$, then there exists a conservative strong Markov process $(X_t,\mc{F}_t,\mbb{P}^x; x \in \mbb{R}^d, t \geq 0)$ such that $(X_t)_{t \geq 0}$ is a weak solution to \eqref{cont-eq55} with respect to $\mbb{P}^{\mu} := \int \mbb{P}^x \, \mu(dx)$ for any $\mu \in \mc{P}(\mbb{R}^d)$.
\end{kor}

The assumption on the existence of a weak solution to \eqref{cont-eq55} is, in particular, satisfied if $b$ and $\sigma$ are continuous. For L\'evy-driven SDEs with discontinuous coefficients we refer to Theorem~\ref{exi-11} for a sufficient condition for the existence. Note that \eqref{cont-eq55} covers SDEs of the form \begin{equation}
	dX_t = b(X_{t-}) \, dt + f(X_{t-}) \, dB_t + g(X_{t-}) \, dJ_t \label{cont-eq5}
\end{equation}
where $(J_t)_{t \geq 0}$ is a pure jump L\'evy process and $(B_t)_{t \geq 0}$ an independent Brownian motion; simply choose $L_t = (B_t,J_t)$ in \eqref{cont-eq55}. Let us mention that Anulova \& Pragarauskas \cite{anu77} proved a Markovian selection theorem for SDEs \eqref{cont-eq5} for the particular case that $f$ is uniformly elliptic.

\begin{proof}[Proof of Corollary~\ref{cont-3}]
	Set $q(x,\xi) := -ib(x) \cdot \xi + \psi(\sigma(x)^T \cdot \xi)$, $x,\xi \in \mbb{R}^k$, and denote by $A$ the pseudo-differential operator with symbol $q$. Since $b$ and $\sigma$ are of sublinear growth, it follows easily that $q$ satisfies the assumptions of Theorem~\ref{cont-1}, and therefore there exists a conservative strongly Markovian solution to the $(A,C_c^{\infty}(\mbb{R}^d))$-martingale problem. It is known, see e.\,g.\ \cite{kurtz}, that any solution to the $(A,C_c^{\infty}(\mbb{R}^d))$-martingale problem with initial distribution $\mu$ is a weak solution to \eqref{cont-eq55}; this finishes the proof.
\end{proof}

\begin{kor}[Stable-dominated processes] \label{cont-5}
	Let $\kappa: \mbb{R}^d \times \mbb{R}^d \backslash \{0\} \to (0,\infty)$ be a mapping such that $x \mapsto \kappa(x,y)$ is continuous for all $y \in \mbb{R}^d \backslash \{0\}$. If there exist finite constants $c_1,c_2>0$ and $\alpha, \beta \in (0,2)$ such that \begin{equation*}
		\kappa(x,y) \leq c_1 \frac{1}{|y|^{d+\alpha}} \I_{\{|y|<1\}} + c_2 \frac{1}{|y|^{d+\beta}} \I_{\{|y| \geq 1\}} \fa x \in \mbb{R}^d, \, \, y \in \mbb{R}^d \backslash \{0\},
	\end{equation*}
	then there exists a conservative strongly Markovian solution to the martingale problem for the pseudo-differential operator with symbol \begin{equation*}
		q(x,\xi) := \int_{\mbb{R}^d \backslash \{0\}} \left( 1-e^{iy \cdot \xi} + iy \cdot \xi \I_{(0,1)}(|y|) \right) \kappa(x,y) \, dy, \qquad x, \xi \in \mbb{R}^d.
	\end{equation*}
\end{kor}

We will see in Section~\ref{har} that Corollary~\ref{cont-5} can be used to establish a Harnack inequality. Corollary~\ref{cont-5} applies, in particular, to stable-like processes. If we choose, for instance, $\kappa(x,y) = |y|^{-d-\alpha(x)}$ for a continuous mapping $\alpha: \mbb{R}^d \to (0,2)$ satisfying $\inf_x \alpha(x)>0$, we find that there exists a strongly Markovian solution to the the martingale problem for the pseudo-differential operator with symbol $q(x,\xi) = |\xi|^{\alpha(x)}$.

\begin{proof}[Proof of Corollary~\ref{cont-5}]
	Using the elementary estimates \begin{equation*}
		|1-e^{iz}+iz| \leq \frac{1}{2} |z|^2 \quad \text{and} \quad |1-e^{iz}| \leq \min\{2,|z|\}
	\end{equation*}
	it is not difficult to see that $q$ has bounded coefficients and satisfies the continuity condition \eqref{cont-eq3}. Applying Corollary~\ref{cont-2} proves the assertion.
\end{proof}

The next example shows that the Markovian selection from Theorem~\ref{cont-1} fails, in general, to be unique. Moreover, it shows that we can, in general, not choose the Markovian selection in such a way that the associated semigroup has nice mapping properties (e.\,g.\ the Feller property or the $C_b$-Feller property). 

\begin{bsp} \label{cont-7}
	Consider the martingale problem for the pseudo-differential operator $A$ with symbol $q(x,\xi) = -2 i \xi \sgn(x) \sqrt{|x|}$. Clearly, (the distribution of) a process $(X_t)_{t \geq 0}$ is a solution to the $(A,C_c^{\infty}(\mbb{R}))$-martingale problem with initial distribution $\mu=\delta_x$ if $(X_t)_{t \geq 0}$ satisfies the ordinary differential equation \begin{equation}
		dX_t = 2 \sgn(X_t) \sqrt{|X_t|} \, dt, \qquad X_0 = x. \label{cont-eq9}
	\end{equation}
	It is not difficult to check that both \begin{equation*}
		X_t := \begin{cases} (t+\sqrt{x})^2, & x \geq 0, \\ - (t+\sqrt{-x})^2, & x<0 \end{cases} \quad \text{and} \quad Y_t := \begin{cases} (t+\sqrt{x})^2, & x>0, \\ - (t+\sqrt{x})^2, & x \leq 0 \end{cases}
	\end{equation*}
	are Markovian solutions, and hence uniqueness of a Markovian selection fails. Moreover, we note that \eqref{cont-eq9} has a unique solution for any $x \neq 0$, and therefore it follows easily that $\lim_{x \downarrow 0} \mbb{E}^x f(X_t) = f(t^2)$ and $\lim_{x \uparrow 0} \mbb{E}^x f(X_t) = f(-t^2)$ for any selection $(X_t,\mbb{P}^x)_{t \geq 0,x \in \mbb{R}^d}$ of solutions to the $(A,C_c^{\infty}(\mbb{R}))$-martingale problem; in particular, the $C_b$-Feller property and the Feller property fail to hold for any Markovian selection.
\end{bsp}

\section{Applications} \label{ana}

In this section we present two applications of Markovian selection theorems in the theory of non-local operators and equations.  The first one is a Harnack inequality for pseudo-differential operators of variable order, cf.\ Section~\ref{har}, and the second one concerns viscosity solutions to a certain integro-differential equation, cf.\ Section~\ref{vis}. 

\subsection{Harnack inequality for non-local operators of variable order} \label{har}

Harnack inequalities are an important tool in the study of partial differential equations. In the last years there has been an increasing interest in Harnack inequalities for functions that are harmonic with respect to a L\'evy-type operator.  Due to the non-local nature of these operators, it is not possible to use the same techniques as for differential operators. It has turned out that the probabilistic approach via martingale problems is very powerful. In order to use this method, it is, however, necessary to know that there exists a strongly Markovian solution to the martingale problem, and many important contributions, e.\,g.\ \cite{bass,levin,kass,wada}, have to \emph{assume} the existence of a strongly Markovian solution. It is, in general, difficult to prove that the martingale problem is well-posed, and this made it so far difficult to check this assumption. Our Markovian selection theorem, Theorem~\ref{cont-1}, allows us to prove the existence without the much harder well-posedness of the martingale problem. \par
In this section we combine the Markovian selection theorem with the results from \cite{bass} to prove a Harnack inequality for operators of variable order. In \cite{bass} Bass and Ka{\ss}mann established a Harnack inequality for pseudo-differential operators of the form \begin{equation*}
	Au(x) = \int_{\mbb{R}^d \backslash \{0\}} (u(x+y)-u(x)-\nabla u(x) \cdot y \I_{(0,1)}(|y|)) \, \kappa(x,y) \, dy, \qquad u \in C_c^{\infty}(\mbb{R}^d), \, \, x \in \mbb{R}^d;
\end{equation*}
their result requires only weak assumptions on the kernel $\kappa$, but for the proof they have to assume that there exists a strongly Markovian solution to the $(A,C_c^{\infty}(\mbb{R}^d))$-martingale problem. Thanks to the Markovian selection theorem, we can give mild assumptions which ensure the existence of such a strongly Markovian solution. \par
We recall the following definition. As usual, $(X_t)_{t \geq 0}$ denotes the canonical process with canonical filtration $\mc{F}_t := \sigma(X_s; s \leq t)$.

\begin{defn} \label{har-3}
	For a linear operator $A: \mc{D}(A) \to \mc{B}_b(\mbb{R}^d)$ and $x \in \mbb{R}^d$ let $\mbb{P}^x$ be a solution to the $(A,\mc{D}(A))$-martingale problem with initial distribution $\delta_x$. A function $u \in \mc{B}_b(\mbb{R}^d)$ is called \emph{harmonic in an open set $D \subseteq \mbb{R}^d$} if $(u(X_{t \wedge \tau_D}))_{t \geq 0}$ is a $\mbb{P}^x$-martingale for each $x \in D$; here \begin{equation*}
		\tau_D := \inf\{t>0; X_t \notin D\}
	\end{equation*}
	denotes the first exit time of $D$.
\end{defn}

The following theorem is the main result in this section.

\begin{thm} \label{har-1}
	Let $\kappa: \mbb{R}^d \times  \mbb{R}^d \backslash \{0\} \to (0,\infty)$ be a Borel measurable mapping such that $x \mapsto \kappa(x,y)$ is continuous for each $y \in \mbb{R}^d \backslash \{0\}$. Assume that there exist constants $c_1,c_2,c_3,c_4>0$ and $\alpha,\beta,\kappa \in (0,2)$ such that the following conditions are satisfied. 
	\begin{enumerate}[label*=\upshape (H\arabic*),ref=\upshape H\arabic*] 
		\item\label{H1} $\kappa(x,y) \leq c_1 |y|^{-d-\kappa}$ for all $x \in \mbb{R}^d$, $|y| > 2$,
		\item\label{H2} $c_2 |y|^{-d-\alpha} \leq \kappa(x,y) \leq c_3 |y|^{-d-\beta}$ for all $x \in \mbb{R}^d$, $0<|y| \leq 2$
		\item\label{H3} $\kappa(x,x-z) \leq c_4 \kappa(y,y-z)$ for all $|x-y| \leq 1$, $|x-z| \geq 1$, $|y-z| \geq 1$.
	\end{enumerate}
	Then the following statements hold for the pseudo-differential operator $A$ with symbol 
	\begin{equation*}
			q(x,\xi) := \int_{\mbb{R}^d \backslash \{0\}} \left( 1-e^{iy \cdot \xi} + iy \cdot \xi \I_{(0,1)}(|y|) \right) \, \kappa(x,y) \, dy, \qquad x,\xi \in \mbb{R}^d,
		\end{equation*}
	\begin{enumerate}[label={\upshape(\roman*)}]
		\item\label{har-1-i} There exists a conservative strong Markov process $(X_t,\mc{F}_t,\mbb{P}^x; x \in \mbb{R}^d,t \geq 0)$ such that $\mbb{P}^x$ solves for each $x \in \mbb{R}^d$ the $(A,C_c^{\infty}(\mbb{R}^d))$-martingale problem with initial distribution $\delta_x$.
		\item\label{har-1-ii} If $\beta-\alpha < 1$ then a Harnack inequality holds; more precisely, if $u \in \mc{B}_b(\mbb{R}^d)$, $u \geq 0$, is harmonic on an open ball $B(x_0,2r)$, then there exists $C>0$ such that \begin{equation}
			\forall x,y \in B(x_0,r): \quad u(x) \leq C u(y); \label{harnack}
		\end{equation}
		the constant $C$ depends on $r$ and $c_1,\ldots,c_4$, but not $x_0$ and $u$.
	\end{enumerate}
\end{thm}

Theorem~\ref{har-1} is a direct consequence of Corollary~\ref{cont-5} and \cite{bass}. Let us point out that the assumptions \eqref{H2}, \eqref{H3} and $\beta-\alpha<1$ are taken from \cite{bass} and are thus needed to prove the Harnack inequality. The condition on the large jumps \eqref{H1} is slightly stronger than in \cite{bass}; this is the price which we have to pay for the existence of a strongly Markovian solution to the martingale problem (more precisely, \eqref{H1} is needed to ensure that the continuity condition \eqref{cont-eq3} holds). \par
Note that the Harnack inequality \eqref{harnack} can be used to study the regularity of harmonic functions, see \cite[Section 4]{kass}.

\begin{bem} \label{har-5}
It is not difficult to see that a function $u \in C_b^2(\mbb{R}^d)$ is harmonic in $B(x_0,2r)$ if $Au(x)=0$ for all $x \in B(x_0,2r)$. Theorem~\ref{har-1}\eqref{har-1-ii} shows, in particular, that any such function $u \geq 0$ satisfies the Harnack inequality \begin{equation*}
	\sup_{x \in B(x_0,r)} u(x) \leq C \inf_{x \in B(x_0,r)} u(x)
\end{equation*}
for some finite constant $C=C(r,c_1,c_2,c_3,c_4)>0$ not depending on $u$ and $x_0$.  \end{bem}

\subsection{Viscosity solutions} \label{vis}

Viscosity solutions were originally introduced by Lions \& Crandall to study non-linear PDEs of the form \begin{equation*}
	F(x,u(x),\nabla u(x),\nabla^2 u(x))=0. 
\end{equation*}
The concept has been successfully to extended to nonlinear non-local equations \begin{equation}
	G(x,u(x),\nabla u(x),\nabla^2 u(x), u(\cdot))=0, \label{vis-eq3}
\end{equation}
and over the last two decades viscosity solutions have turned out to be one of the most important notions for generalized solutions. Non-linear non-local equations \eqref{vis-eq3} appear naturally in the theory of stochastic processes, for instance in the study of Feller processes (cf.\ \cite{ltp}) and sublinear Markov processes (see e.\,g.\ \cite{julian,denk} and the references therein). Costantini \& Kurtz \cite{cost} showed that there is a close connection between martingale problems and viscosity solutions to the integro-differential equation \begin{equation}
	\lambda u(x)-Au(x) = f(x), \qquad \lambda>0, f \in C_b(\mbb{R}^d), \label{vis-eq1}
\end{equation}
roughly speaking, they showed that a comparison principle for \eqref{vis-eq1} implies the well-posedness of the martingale problem for the operator $A$. Recently, this result has been used by Zhao \cite{zhao} to derive the existence of a unique weak solution to the SDE \begin{equation*}
	dX_t = b(X_{t-}) \, dt + dL_t
\end{equation*}
driven by an $\alpha$-stable L\'evy process, $\alpha \in (0,1)$, under weak regularity assumptions on the drift $b$. We will use the Markovian selection theorem, Theorem~\ref{cont-1}, to give a sufficient condition which ensures that the function \begin{equation*}
	u(x) := \sup_{\mbb{P} \in \Pi_x} \mbb{E}_{\mbb{P}} \left( \int_{(0,\infty)} e^{-\lambda t} f(X_t) \, dt \right)
\end{equation*}
is a viscosity solution to \eqref{vis-eq1}; here $(X_t)_{t \geq 0}$ is the canonical process and $\Pi_x$ is the family of probability measures on $\Omega = D[0,\infty)$ which solve the $(A,C_c^{\infty}(\mbb{R}^d))$-martingale problem with initial distribution $\delta_x$. 

\begin{defn}  \label{vis-3}
	Let $A:  C_b^2(\mbb{R}^d) \to C_b(\mbb{R}^d)$ be a linear operator, $\lambda>0$ and $f \in C_{\infty}(\mbb{R}^d)$. \begin{enumerate}
		\item An upper semicontinuous bounded function $u$ is a \emph{viscosity subsolution} to $\lambda u-Au = f$ if the implication \begin{equation*}
			\sup_{x \in \mbb{R}^d} (u(x)-\phi(x)) = u(x_0)-\phi(x_0) \implies \lambda u(x_0)-A\phi(x_0) \leq f(x_0)
		\end{equation*}
		holds for any $\phi \in C_b^2(\mbb{R}^d)$, $x_0 \in \mbb{R}^d$.
		\item A lower semicontinuous bounded function $v$ is a \emph{viscosity supersolution} to $\lambda u-Au = f$ if \begin{equation*}
			\inf_{x \in \mbb{R}^d} (v(x)-\phi(x)) = v(x_0)-\phi(x_0) \implies v(x_0)- A \phi(x_0) \geq f(x_0)
		\end{equation*}
		for any $\phi \in C_b^2(\mbb{R}^d)$, $x_0 \in \mbb{R}^d$.
		\item A function $u \in C_b(\mbb{R}^d)$ is a \emph{viscosity solution} to $\lambda u-Au=f$ if $u$ is both a viscosity subsolution and a viscosity supersolution.
	\end{enumerate}
\end{defn}

For a pseudo-differential operator $A$ with negative definite symbol $q$, the assumption $A(C_b^2(\mbb{R}^d)) \subseteq C_b(\mbb{R}^d)$ in Definition~\ref{vis-3} means that the symbol $q$ is continuous (with respect to $x$ and $\xi$) and has bounded coefficients.

\begin{thm} \label{vis-5}
	Let $A$ be a pseudo-differential operator with symbol $q$ of the form \eqref{symbol}, $q(x,0)=0$, and let $f \in C_{\infty}(\mbb{R}^d)$, $f \geq 0$. Assume that $x \mapsto q(x,\xi)$ is continuous, $q$ has bounded coefficients and \begin{equation}
		\lim_{R \to \infty} \sup_{|y| \leq R} \sup_{|\xi| \leq R^{-1}} |q(y,\xi)| = 0. \label{vis-eq11}
	\end{equation}
	\begin{enumerate}
		\item\label{vis-5-i} The function \begin{equation*}
			u(x) := \sup_{\mbb{P} \in \Pi_x} \mbb{E}_{\mbb{P}} \left( \int_{(0,\infty)} e^{-\lambda t} f(X_t) \, dt \right), \qquad x \in \mbb{R}^d
		\end{equation*}
		is a viscosity subsolution to \begin{equation*}
			\lambda u-Au = f;
		\end{equation*}
		here $\Pi_x$ denotes the set of probability measures on $D[0,\infty)$ which are a solution to the $(A,C_c^{\infty}(\mbb{R}^d))$-martingale problem with initial distribution $\delta_x$.
		\item\label{vis-5-ii} If $u$ is lower semicontinuous, then $u$ is a viscosity solution to $\lambda u-Au = f$.
	\end{enumerate}
\end{thm}

If $u$ is lower semicontinuous (hence, by \eqref{vis-5-i}, continuous), then a result by Barles et al. \cite{barles} shows that - under rather general assumptions -- $u$ is H\"{o}lder continuous. This, in turn, would allow us to derive new results on the well-posedness of martingale problems using a similar approach as in \cite{zhao}. It seems, however, that the lower semicontinuity of $u$ is, in general, difficult to check; in fact, Example~\ref{cont-7} shows that $u$ fails, in general, to be (lower semi)continuous. 

\begin{proof}[Proof of Theorem~\ref{vis-5}]
	The first statement is a direct consequence of \cite[Lemma 3.5]{cost} and Proposition~\ref{cont-9} but we prefer to give a direct proof which gives both \eqref{vis-5-i} and \eqref{vis-5-ii}. By Theorem~\ref{exi-1} we have $\Pi_x \neq \emptyset$ and therefore $u(x) \in \mbb{R}$ is well-defined for each $x \in \mbb{R}^d$. If $(x_n)_{n \in \mbb{N}} \subseteq \mbb{R}^d$ is such that $x_n \to x \in \mbb{R}^d$, then Proposition~\ref{cont-9} shows that for any $\mbb{P}_n \in \Pi_{x_n}$, $n \geq 1$, the sequence $(\mbb{P}_n)_{n \in \mbb{N}}$ is tight. It is not difficult to see that this implies that $u$ is upper semicontinuous, see \cite[Lemma 3.4]{cost} for more details.  \par
	Now let $f \in C_{\infty}(\mbb{R}^d)$ and $\phi \in C_b^2(\mbb{R}^d)$.  By Corollary~\ref{cont-2} there exists $\mbb{P}^x \in \Pi_x$, $x \in \mbb{R}^d$, such that $(X_t,\mc{F}_t,\mbb{P}^x; x \in \mbb{R}^d, t \geq 0)$ is a strong Markov process with respect to the canonical filtration $\mc{F}_t := \sigma(X_s; s \leq t)$ and \begin{equation}
		u(x) = \mbb{E}^x\left( \int_0^{\infty} e^{-\lambda t} f(X_t) \, dt \right) \fa x \in \mbb{R}^d; \label{vis-eq17}
	\end{equation}
	here (and throughout the remaining part of the proof) we use the shorthand $\mbb{E}^x := \mbb{E}_{\mbb{P}^x}$.	Using a standard approximation procedure and the fact that $\|Af\|_{(2)} \leq c \|f\|_{(2)}$, $f \in C_b^2(\mbb{R}^d)$, for some absolute constant $c>0$, it follows easily that $\mbb{P}^x$ is a solution to the $(A,C_b^2(\mbb{R}^d))$-martingale problem with initial distribution $\delta_x$. Consequently, \begin{equation*}
		\phi(X_t)-\phi(X_0)- \int_0^t A \phi(X_s) \, ds
	\end{equation*}
	is a $\mbb{P}^x$-martingale, and this implies that \begin{equation}
		\phi(x) = \int_{(0,\infty)} e^{-\lambda t} \mbb{E}^x(\lambda \phi(X_t)-A \phi(X_t)) \, dt, \label{vis-eq19}
	\end{equation}
	cf.\ \cite[Lemma 2.9]{cost}. Thus, by \eqref{vis-eq17} and \eqref{vis-eq19}, \begin{equation*}
		u(x)-\phi(x) = \int_{(0,\infty)} e^{-\lambda t} \mbb{E}^x \big( f(X_t)- \lambda \phi(X_t)+ A \phi(X_t) \big) \, dt.
	\end{equation*}
	If we define \begin{equation*}
		\tau_r^x := \tau_r := \min\{t>0; |X_t-x|>r\} \wedge r
	\end{equation*}
	then we find from the strong Markov property of $(X_t)_{t \geq 0}$ that \begin{align*}
		u(x)-\phi(x)
		&= \mbb{E}^x \left( \int_{(0,\tau_r)} e^{-\lambda t} [f(X_t)-\lambda \phi(X_t) + A \phi(X_t)] \, dt \right) \\
		&\quad + \mbb{E}^x \left( e^{-\lambda \tau_r} \mbb{E}^{X_{\tau_r}} \left[ \int_{(0,\infty)} (f(X_t)- \lambda \phi(X_t) + A \phi(X_t)) \, dt \right] \right).
	\end{align*}
	Invoking \eqref{vis-eq17} and \eqref{vis-eq19} we find \begin{align} \begin{aligned} 
		u(x) - \phi(x)
		&= \mbb{E}^x \left( \int_{(0,\tau_r)} e^{-\lambda t} [f(X_t)-\lambda \phi(X_t) + A \phi(X_t)] \, dt \right) \\
		&\quad + \mbb{E}^x \big( e^{-\lambda \tau_r} [u(X_{\tau_r})- \varphi(X_{\tau_r})] \big). \end{aligned} \label{vis-eq21}
	\end{align}
	Now let $\phi \in C_b^2(\mbb{R}^d)$ and $x_0 \in \mbb{R}^d$ be such that \begin{equation*}
		\sup_{x \in \mbb{R}^d} (u(x)-\phi(x)) = u(x_0)-\phi(x_0).
	\end{equation*}
	Without loss of generality, we may assume $u(x_0) =\phi(x_0)$. Then $u \leq \phi$ yields $u(X_{\tau_r})-\varphi(X_{\tau_r}) \leq 0$, and so \begin{equation*}
		0 = u(x_0)-\phi(x_0) \leq  \mbb{E}^{x_0} \left( \int_{(0,\tau_r)} e^{-\lambda t} [f(X_t)-\lambda \phi(X_t) + A \phi(X_t)] \, dt \right).
	\end{equation*}
	By the right-continuity of the sample paths of $(X_t)_{t \geq 0}$ we have $\mbb{E}^{x_0}(\tau_r)>0$ for $r>0$ sufficiently small; moreover, trivially $\mbb{E}^{x_0}(\tau_r) \leq r < \infty$. Dividing both sides of the previous equation by $\mbb{E}_{x_0} \tau_r$ and letting $r \to 0$ we obtain \begin{equation*}
		0 \leq f(x_0)- \lambda \phi(x_0) + A \phi(x_0).
	\end{equation*}
	As $\phi(x_0) = u(x_0)$, this shows that $u$ is a viscosity subsolution. On the other hand, if \begin{equation*}
		\inf_{x \in \mbb{R}^d} (u(x)-\phi(x)) = u(x_0)-\phi(x_0) = 0,
	\end{equation*}
	then \eqref{vis-eq21} gives \begin{align*}
		0 = u(x_0)-\phi(x_0) \geq  \mbb{E}^{x_0} \left( \int_{(0,\tau_r)} e^{-\lambda t} [f(X_t)-\lambda \phi(X_t) + A \phi(X_t)] \, dt \right),
	\end{align*}
	and we conclude that \begin{equation*}
		0 \geq f(x_0)- \lambda \phi(x_0)+ A \phi(x_0) = f(x_0)-\lambda u(x_0) + A \phi(x_0);
	\end{equation*}
	this proves \eqref{vis-5-ii}.
\end{proof}

\appendix

\section{}

In the first part of the proof of Corollary~\ref{exi-13} we used heat kernel estimates from \cite{matters} to establish the Krylov estimate \eqref{exi-eq35}. Let us explain in more detail how to obtain the required estimates;  \cite[Theorem 3.8]{matters} gives heat kernel estimates for the transition densities of Feller processes with symbols of the form \begin{equation*}
	q(x,\xi) := \psi_{h(x)}(\xi), \qquad x,\xi \in \mbb{R}^d;
\end{equation*}
here $(\psi_{\kappa})_{\kappa \in I}$, $I \subseteq \mbb{R}^k$, is a family of continuous negative functions and $h: \mbb{R}^d \to I$ a H\"older continuous function. In the proof of Corollary~\ref{exi-13} (Step 1) we are interested in the particular case that \begin{equation}
	\psi_{\kappa}(\xi) = i \xi \kappa_1 + \psi(\kappa_2 \xi), \qquad  h(x) := \begin{pmatrix} f(x) \\ g(x) \end{pmatrix} \qquad x,\xi \in \mbb{R},  \label{app-eq11}
\end{equation}
where $\psi$ is the characteristic exponent of the driving L\'evy process $(L_t)_{t \geq 0}$, $f: \mbb{R} \to \mbb{R}$, $g: \mbb{R} \to (0,\infty)$ are H\"older continuous functions such that \begin{equation*}
	f^L := \inf_{x \in \mbb{R}^d} f(x) \leq  \sup_{x \in \mbb{R}^d} f(x) =: f^U < \infty \qquad 0<g^L := \inf_{x \in \mbb{R}^d} g(x) \leq \sup_{x \in \mbb{R}^d} g(x) =: g^U < \infty
\end{equation*}
and $\kappa:=(\kappa_1,\kappa_2) \in I := [f^L,f^U] \times [g^L,g^U]$. The assumptions on the L\'evy process $(L_t)_{t \geq 0}$ ensure that \cite[Theorem 3.8]{matters} is indeed applicable. In order to state the heat kernel estimates we have to recall some assumptions on the family $(\psi_{\kappa})_{\kappa \in I}$: \begin{enumerate}
	\item There exists a H\"older continuous function $\gamma_{\infty}: I \to (0,2]$ and constants $c_1,c_2>0$ such that $\gamma_{\infty}^L := \inf_{\kappa \in I} \gamma_{\infty}(\kappa)>0$ and \begin{equation*}
	\re \psi_{\kappa}(\xi) \geq c_1 |\xi|^{\gamma_{\infty}(\kappa)} \quad \text{and} \quad |\psi_{\kappa}(\xi)| \leq c_2 |\xi|^{\gamma_{\infty}(\kappa)} \fa   |\xi| \gg 1, \, \, \kappa \in I.
\end{equation*}
	\item There exists a measurable mapping $\gamma_0: I \to (0,2]$ and a constant $c_3>0$ such that $\gamma_0^L := \inf_{\kappa \in I} \gamma_0(\kappa)>0$ and \begin{equation*}
		|\psi_{\kappa}(\xi)| \leq c_3 |\xi|^{\gamma_0(\kappa)} \fa \kappa \in I,  \, \, |\xi| \leq 1.
	\end{equation*}
\end{enumerate}
\cite[Theorem 3.8]{matters} states that, under suitable further assumptions on $(\psi_{\kappa})_{\kappa \in I}$ (all of them are satisfied in the proof of Corollary~\ref{exi-13}, Step 1), there exists a constant $C_1>0$ such that the transition density $p$ of the Feller process with symbol $q(x,\xi) = \psi_{h(x)}(\xi)$ satisfies \begin{equation}
	|p(t,x,y)| \leq C_1 S(x-y,h(y),t) + C_1 \frac{1}{1+|x-y|^{d+\gamma_0^L \wedge \gamma_{\infty}^L}}, \qquad x,y \in \mbb{R}^d, t \in (0,T], \label{app-eq13}
\end{equation}
where \begin{equation*}
	S(z,\kappa,t) := \begin{cases} t^{-d/\gamma_{\infty}(\kappa)}, & |z| \leq t^{1/\gamma_{\infty}(\kappa)} \wedge 1, \\ \frac{t}{|z|^{d+\gamma_{\infty}(\kappa)}}, & t^{1/\gamma_{\infty}(\kappa)} < |z| \leq 1, \\ \frac{t}{|z|^{d+ \gamma_{\infty}(\kappa) \wedge \gamma_0(\kappa)}}, & |z| > 1. \end{cases}
\end{equation*}
Integrating \eqref{app-eq13} with respect to $t$, it follows easily that \begin{equation}
	\int_0^T |p(t,x,y)| \, dt \leq C_2 Q(x-y), \qquad x,y \in \mbb{R}^d \label{app-eq15}
\end{equation}
for some constant $C_2>0$ and \begin{equation*}
	Q(z) := |z|^{-d- \gamma_0^L \wedge \gamma_{\infty}^L} \I_{\{|z| \geq 1\}} + (1+|\log|z|| + |z|^{-d+\gamma_{\infty}^L}) \I_{\{0<|z|<1\}}+ \I_{\{z=0\}}.
\end{equation*}
A close look at the proof of \cite[Theorem 3.8]{matters} shows that there exists a continuous function $F: (0,\infty)^{10} \to (0,\infty)$ such that \begin{equation}
	C_2 =  F \left( d, T, c_2,c_3, \|h\|_{\varrho(h)}, \frac{1}{\varrho(\gamma_{\infty} \circ h)},\frac{1}{\|\gamma_{\infty} \circ h\|_{\varrho(\gamma_{\infty} \circ h)}}, \frac{1}{\gamma_0^L}, \frac{1}{\gamma_{\infty}^L}, \frac{1}{c_1} \right); \label{app-eq17}
\end{equation}
here $\varrho(h)$ and $\varrho(\gamma_{\infty} \circ h)$ denote the H\"older exponent of $h$ and $\gamma_{\infty} \circ h$, respectively, and $\|\cdot\|_{\varrho}$ is the H\"older norm. This means, for instance, that the constant $C_2$ is, in general, going blow up if the H\"older exponent of $\gamma_{\infty} \circ h$ tends to $0$.  For the particular case we consider in the proof of Corollary~\ref{exi-13}, cf.\ \eqref{app-eq11}, we have \begin{align*}
	\gamma_{\infty}(\kappa_1,\kappa_2) &= \begin{cases} \beta, & f^U = f^L = 0, \\ \max\{1,\beta\}, & \text{otherwise} \end{cases} \\
	\gamma_0(\kappa_1,\kappa_2) &= \begin{cases} \alpha, & f^U = f^L = 0, \\ \min\{1,\alpha\}, & \text{otherwise} \end{cases} \\
	c_1 &:= g^L \quad c_2 := c_3 := f^U + g^U
\end{align*}
where $\alpha,\beta \in (0,2]$ denote constants from \eqref{L1}-\eqref{L3}, cf.\ Corollary~\ref{exi-13}. Note that $\gamma_{\infty}$ and $\gamma_0$ do not depend on $\kappa=(\kappa_1,\kappa_2)$, and therefore $\gamma_{\infty} \circ h$ is Lipschitz continuous for \emph{any} mapping $h$. Hence, \begin{equation*}
	C_2 = \tilde{F} \left( T, g^U, f^U, \|f\|_{\varrho(f)}, \|g\|_{\varrho(g)},\frac{1}{g^L} \right)
\end{equation*}
for some continuous function $\tilde{F}$. By \eqref{app-eq17} this gives \eqref{exi-eq34}, and hence \eqref{exi-eq35}. \par \medskip

For the proof of Corollary~\ref{exi-13} we also used the following result which concerns the approximation of Borel measurable functions by H\"older continuous functions. 

\begin{lem} \label{app-1}
	Let $f:\mathbb{R}^d \to \mathbb{R}$ be a measurable bounded function. Then there exist sequences $(f_n)_{n \in \mbb{N}} \subseteq C^{\infty}(\mbb{R}^d)$ and $(\alpha_n)_{n \in \mbb{N}} \subseteq (0,1)$ such that \begin{enumerate}
			\item $f_n$ is $\alpha_n$-H\"{o}lder continuous for each $n \in \mbb{N}$ and $$M := \sup_{n \in \mbb{N}} \|f_n\|_{\alpha_n} < \infty,$$
			\item $f_n(x) \to f(x)$ for Lebesgue almost all $x \in \mbb{R}^d$,
			\item $\|f_n\|_{\infty} \leq \|f\|_{\infty}$ and $f_n(x) \geq \inf_{y \in \mbb{R}^d} f(y)$ for each $n \in \mbb{N}$, $x \in \mbb{R}^d$.
		\end{enumerate}
\end{lem}

\begin{proof}
	\textbf{Step 1:} For any Lipschitz continuous function $f: \mbb{R}^d \to \mathbb{R}$ there exists $\alpha>0$ such that $\|f\|_{\alpha} \leq 4\|f\|_{\infty}$. \par
	\emph{Indeed:} Denote by $L>0$ the Lipschitz constant of $f$ and choose $\alpha>0$ sufficiently small such that $2L^{\alpha} \leq 3$. If $|x-y| \geq 1/L$, then \begin{equation*}
			\frac{|f(x)-f(y)|}{|x-y|^{\alpha}} \leq 2L^{\alpha} \|f\|_{\infty} \leq 3 \|f\|_{\infty}.
		\end{equation*}
	If $|x-y| < 1/L$, then \begin{equation*}
		\frac{|f(x)-f(y)|}{|x-y|^{\alpha}} \leq L |x-y|^{1-\alpha} \leq L^{\alpha} \leq 3 \|f\|_{\infty}.
	\end{equation*}
	\textbf{Step 2:} Let $\chi \in C_c^{\infty}(\mbb{R}^d)$ be such that $0 \leq \chi \leq 1$ and $\int_{\mbb{R}^d} \chi(y) \, dy = 1$. If we set $\chi_n(x) := 1/n^d \chi(x/n)$ and $f_n := f \ast \chi_n \in C^{\infty}(\mbb{R}^d)$, then $f_n \to f$ (Lebesgue)almost everywhere. Moreover, \begin{equation*}
			\|f_n\|_{\infty} \leq \|f\|_{\infty}
	\end{equation*}
	and therefore it follows from Step 1 that we can choose $\alpha_n>0$ such that $\|f_n\|_{\alpha_n} \leq 4 \|f\|_{\infty}$ for all $n \in \mbb{N}$. Finally, \begin{equation*}
			f_n(x) = \int f(y) \chi_n(y-x) \, dy \geq \left( \inf_{y \in \mbb{R}^d} f(y) \right) \int \chi_n(y-x) \, dy = \inf_{y \in \mbb{R}^d} f(y). \qedhere
		\end{equation*}
\end{proof}		

\begin{ack}
	The proof of Lemma~\ref{app-1} goes back to David C.\ Ullrich, and the author gratefully acknowledges this contribution. Moreover, the author would like to thank Niels Jacob and Ren\'e Schilling for their valuable remarks and suggestions. The author also thanks the \emph{Institut national des sciences appliquées de Toulouse} for its hospitality during her stay in Toulouse, where a part of this work was accomplished.
\end{ack}

\end{document}